\documentclass{article} \usepackage{amsmath,amsfonts,amsthm}
\newtheorem{thm}{Theorem}[section] \newtheorem{pro}[thm]{Proposition}
\newtheorem{cor}[thm]{Corollary} \newtheorem{lem}[thm]{Lemma}

 \newtheorem{ex}[thm]{Example}
\newtheorem{rem}[thm]{Remark} \newtheorem{defn}[thm]{Definition}
 \newcommand{\dime}{\operatorname{dim}}

\newcommand{\res}{\operatorname{res}}
\newcommand{\Term}{\operatorname{Term}}
\newcommand{\Init}{\operatorname{Init}}
\newcommand{\spec}{\operatorname{spec}}
\newcommand{\Endo}{\operatorname{End}}
\newcommand{\Supp}{\operatorname{Supp}}
\newcommand{\Order}{\operatorname{ord}}
\newcommand{\kerl}{\operatorname{Ker}}
\newcommand{\Img}{\operatorname{Image}} \def\RR{\mathbf{R}}
 \def\div{\raise 1pt \hbox{\big|}}

\begin{document}

\title{On generalized Witt algebras in one variable}
\author{Ki-Bong Nam and Jonathan Pakianathan}
\maketitle

\begin{abstract}
We study a class of infinite dimensional Lie algebras called 
generalized Witt algebras (in one variable). 
These include the classical Witt algebra and
the centerless Virasoro algebra as important examples.

We show that any such generalized Witt algebra is a semisimple, indecomposable 
Lie algebra which does not contain any abelian Lie subalgebras of dimension 
greater than one. 

We develop an invariant of these generalized Witt algebras 
called the spectrum, 
and use it to show that there exist infinite families of nonisomorphic, 
simple, generalized Witt algebras and infinite families of nonisomorphic, 
nonsimple, generalized Witt algebras.

We develop a machinery that can be used to study the endomorphisms of 
a generalized Witt algebra in the case that the spectrum is ``discrete''.
We use this to show, that among other things, every nonzero Lie algebra 
endomorphism of the classical Witt algebra is an automorphism and 
every endomorphism of the centerless Virasoro algebra fixes a canonical 
element up to scalar multiplication. 

However, not every injective 
Lie algebra endomorphism of the centerless 
Virasoro algebra is an automorphism.


\end{abstract}

\section{Introduction}

Throughout this paper, we will work over a field $\mathbf{k}$ of characteristic
zero. Also note that there will be no finiteness 
constraints on the dimension of
the Lie algebras in this paper - in fact, most of the Lie algebras that
we will consider will be infinite dimensional.  

We now sketch the basic results and ideas of this paper in this introductory 
section. Precise definitions 
of the concepts can be found within the paper. 

Let $\RR$ be the field of fractions of the power series algebra 
$\mathbf{k}[[x]]$.

Following~\cite{Kac}, we define a stable algebra to be a subalgebra of 
$\RR$ which 
is closed under formal differentiation $\partial$. Notice that we 
confine ourselves to the one variable case throughout this paper. 

Important examples of stable algebras are the polynomial algebra 
$\mathbf{k}[x]$, the power series algebra $\mathbf{k}[[x]]$ and the 
Laurent polynomial algebra $\mathbf{k}[x,x^{-1}]$. 

Following~\cite{Kaw} and ~\cite{Nam}, to 
every stable algebra $A$, we associate a Lie algebra $Witt(A)$. 
We refer to $Witt(A)$ as a generalized 
Witt algebra.
(The reader is warned, that there are different definitions of what 
a generalized Witt algebra is in the literature. Please look at 
Definition~\ref{defn: Witt} for ours.)

$Witt(\mathbf{k}[x])$ is the classical Witt algebra, (See~\cite{Cou}) and 
$Witt(\mathbf{k}[x,x^{-1}])$ is called the centerless Virasoro algebra 
in the literature. (See~\cite{Kap}.)

A Lie algebra is called self-centralizing if it contains no abelian Lie 
subalgebras of dimension greater than one. 
We prove:

\begin{thm}[Theorem~\ref{thm: sc} and Proposition~\ref{pro: sc}]
Every generalized \\ 
Witt algebra is self-centralizing. 

Furthermore, if 
it is infinite dimensional (which is the case for all but one trivial example 
where $A=\mathbf{k}$), then a generalized Witt algebra must be semisimple and 
indecomposable.
\end{thm}

To contrast, over an algebraically closed field, 
it is shown that the only finite dimensional Lie algebra 
which is self-centralizing, semisimple and indecomposable is 
$\mathfrak{sl}_2$, the Lie algebra of $2 \times 2$ matrices of trace zero.

However a generalized Witt algebra need not be simple, some are and some are 
not.

If a generalized Witt algebra has a nonzero ad-diagonal element, i.e., 
nonzero $\alpha$ such that $ad(\alpha)$ is diagonal in some basis, we 
show that the set of eigenvalues of $ad(\alpha)$ possesses the  
algebraic structure of a pseudomonoid. 

We call this pseudomonoid, the 
spectrum of $\alpha$. We then show in Proposition~\ref{pro: Specwd} 
that any other nonzero ad-diagonal 
element of this Lie algebra, has to have an equivalent spectrum.
This allows us to define the spectrum of $\mathfrak{L}$ to be the spectrum of 
any nonzero ad-diagonal element. It is then shown that this is indeed an 
invariant for these kinds of Lie algebras, i.e., isomorphic Lie algebras 
have equivalent spectra. 

The constraint that the Lie algebra possesses a nonzero ad-diagonal element, 
is not so bad as all the classical examples possess this property. 

In these pseudomonoids, one can define the notion of an ideal subset.
We show:

\begin{pro}[Proposition~\ref{pro: idealcor}]
Let $\mathfrak{L}$ be a generalized Witt algebra with nonzero 
ad-diagonal element 
and let $G$ be its spectrum. Then there is a one-to-one 
correspondence between the ideal subsets of $G$ and the ideals of 
$\mathfrak{L}$.

If $G$ is actually an abelian group then it is simple as a psuedomonoid 
and hence $\mathfrak{L}$ is simple.
\end{pro}

Since the classical Witt algebra and centerless Virasoro algebra have 
nonzero ad-diagonal elements, and their spectra are simple pseudomonoids, 
we recover the well-known fact, that they are simple, as a corollary. 

Using this spectrum invariant, we can distinguish between 
nonisomorphic generalized Witt algebras and show that there is a rich 
variety of such algebras (with nonzero ad-diagonal element): 

\begin{pro}[Examples~\ref{ex: three}, \ref{ex: five} and \ref{ex: six}]
There exist infinite families of nonisomorphic, simple, generalized 
Witt algebras and there exist infinite familes of nonisomorphic, 
nonsimple, generalized Witt algebras.

In fact for every submonoid of $(\mathbf{k},+)$, there is a 
generalized Witt algebra with that monoid as its spectrum.

Thus, in particular since every torsion-free abelian group embeds 
into the additive group of some rational vector space, we may get 
any torsion-free abelian group as the spectrum of a generalized 
Witt algebra in one variable 
by suitable choice of the base field $\mathbf{k}$. 
\end{pro}

A machinery is obtained to find the set of eigenvalues of any element 
in a generalized Witt algebra. It uses formal calculus and in particular,  
the logarithmic derivative. It is stated in Theorem~\ref{thm: spectraltheorem}.

Finally, motivated by~\cite{Rud}, we discuss injective  
Lie algebra endomorphisms of generalized Witt algebras. 

In the case where the generalized Witt algebra 
possesses a ``discrete'' spectrum, one can show that 
such an endomorphism must essentially fix a nonzero ad-diagonal element. 
(See Theorem~\ref{thm: Jacobian}.)

As corollaries of this fact we can easily 
obtain information about endomorphisms of these Lie algebras and prove 
things such as:

\begin{thm}[Corollaries~\ref{cor: classWitt} and \ref{cor: Virasoro}]
Any nonzero Lie algebra endomorphism $f$ of the classical Witt algebra is 
actually an automorphism and furthermore, 
$$
f(x\partial)=(x+b)\partial
$$
for some $b \in \mathbf{k}$.

If $f$ is a nonzero Lie algebra endomorphism of the centerless Virasoro 
algebra, then $f$ is injective and 
$$
f(x\partial)=\frac{1}{a}x\partial
$$
for some nonzero integer $a$. 
However $f$ need not be onto.

More precisely, the centerless Virasoro algebra possesses injective 
Lie algebra endomorphisms which are not automorphisms.
\end{thm}

One should compare this to the Jacobian conjecture for the classical 
Weyl algebra which states that any nonzero algebra endomorphism is an 
automorphism. This conjecture is still open. 
The classical Witt algebra is the Lie algebra of derivations of the 
classical Weyl algebra. (See~\cite{Cou}).

We remark that the automorphisms of the centerless Virasoro algebra 
were known and studied for example in \cite{Dok}.

This completes this introductory overview.

\section{Generalized Weyl algebras}

Let $\mathbf{k}[[x]]$ be the power series algebra over
$\mathbf{k}$, and let $\RR$ be its field of fractions. Note, since
$\mathbf{k}[[x]]$ is a local ring with maximal ideal $(x)$, $\RR$ is
obtained from $\mathbf{k}[[x]]$ by inverting $x$. Thus every element
$g \in \RR$ can be written in the form:

$$
g=\sum_{i=N}^\infty \alpha_i x^i
$$
for suitable $\alpha_i \in \mathbf{k}$ and $N \in \mathbb{Z}$. 

Notice that $\RR$ acts on itself by left multiplication and this gives us
a monomorphism of $\mathbf{k}$ vector spaces:

$$
\tau : \RR \rightarrow \Endo_{\mathbf{k}}(\RR).
$$

Furthermore, there also exists $\partial \in \Endo_{\mathbf{k}}(\RR)$
which corresponds to formal differentiation with respect to $x$, i.e.,
$$
\partial (\sum_{i=N}^\infty \alpha_i x^i) = \sum_{i=N}^\infty i\alpha_i
x^{i-1}.
$$
It is easy to verify that $\partial(g)=0$ if and only if $g$ is a constant.

\begin{defn}
A stable algebra $A$ is a subalgebra of $\RR$
with the property that $\partial(A) \subseteq A$.
\end{defn}

\begin{rem}
Three important examples of stable algebras are the polynomial
algebra $\mathbf{k}[x]$, the power series algebra $\mathbf{k}[[x]]$,
and the Laurent polynomial algebra $\mathbf{k}[x,x^{-1}]$.
(Recall a Laurent polynomial is an element of the form
$ \sum_{i=N}^{M} \alpha_i x^i$ for suitable $N,M \in \mathbb{Z}$ and
$\alpha_i \in \mathbf{k}$.)
\end{rem}

\begin{defn}
Given a stable algebra $A$, we define $Weyl(A)$ to be the subalgebra
of $\Endo_{\mathbf{k}}(\RR)$ generated by $\tau(A)$ and $\partial$. Thus,
$Weyl(A)$ is an associative algebra with identity element equal to 
the identity endomorphism
of $\RR$. We will identify $A$ with its image 
$\tau(A) \subseteq \Endo_{\mathbf{k}}(\RR)$ from
now on.
\end{defn}

\begin{lem}
Let $A$ be a stable algebra. 
For any $f \in A$, one has $ \partial f - f \partial = f' $ in $Weyl(A)$.
Thus for any $\alpha \in Weyl(A)$, one has $\alpha = \sum_{i=0}^{N} \alpha_i
\partial^i$ for suitable $N \in \mathbb{N}$ and $\alpha_i \in A$.

Furthermore, if $\{ e_i | i \in I \}$ is a $\mathbf{k}$-basis for $A$, then 
$\{ e_i \partial^j | i \in I, j \in \mathbb{N} \}$ is a $\mathbf{k}$-basis
for $Weyl(A)$.
\end{lem}
\begin{proof} The proof is standard and is left to the reader.
\end{proof} 

\begin{rem}
$Weyl(\mathbf{k}[x])$ is the classical Weyl algebra. It is a simple
algebra which has no zero divisors, (see~\cite{Cou}). 
In general, one can define an order on $Weyl(\RR)$ such that the
order of a nonzero element is equal to the highest exponent of $\partial$
in its canonical expression and is defined to be $-\infty$ for the 
zero element.

Then one shows that $\Order (\alpha \beta) = \Order (\alpha) + \Order (\beta)$
for any $\alpha, \beta \in Weyl(\RR)$ (see~\cite{Cou}) 
and it easily follows that
$Weyl(\RR)$ has no zero divisors. Hence, $Weyl(A)$,
which is a subalgebra of $Weyl(\RR)$, has no zero divisors in 
general. Note however, that in general, $Weyl(A)$ need not be simple.
\end{rem}

\section{Generalized Witt algebras}

\begin{defn}
\label{defn: Witt}
Let $Witt(A)$ be the subspace of $Weyl(A)$ consisting of the order 1 elements
together with zero. Thus $\alpha \in Witt(A)$ if $\alpha$ can be written
as $f\partial$ for some $f \in A$. 
\end{defn}

It is easy to check that $Witt(A)$ is a Lie subalgebra of $Weyl(A)$. (Note,
it is not a subalgebra of $Weyl(A)$.)

If $\{ e_i\}_{i \in I}$ is a $\mathbf{k}$-basis 
for $A$ then $\{ e_i \partial \}_{i \in I}$
is a $\mathbf{k}$-basis for $Witt(A)$.

Proposition~\ref{pro: otherdefn} shows how our definition is related to 
the one found in \cite{Dok}.

\begin{rem}
$Witt(\mathbf{k}[x])$ is the classical Witt algebra. It is the Lie
algebra of derivations of the classical Weyl algebra (see~\cite{Cou}),
and is a simple Lie algebra. However, in general, $Witt(A)$ is not neccessarily
simple. 
$Witt(\mathbf{k}[x,x^{-1}])$ is called the centerless Virasoro algebra
in the literature. (See~\cite{Kap}.)
\end{rem}

In general, we cannot claim that $Witt(A)$ is simple, but these generalized
Witt algebras do share one important common property - they are 
self-centralizing.

\begin{defn}
Given a Lie algebra $\mathfrak{L}$ and an element $l \in \mathfrak{L}$,
we define the centralizer of $l$, $C(l) = \{ x \in \mathfrak{L} | [l,x]=0 \}$.
Notice, by the Jacobi identity, $C(l)$ is always a Lie subalgebra of
$\mathfrak{L}$ containing $l$.
\end{defn}

\begin{pro}
\label{pro: TFAEsc}
Given a Lie algebra $\mathfrak{L}$, the following conditions are equivalent.
\\
\noindent
(a) For any nonzero $l \in \mathfrak{L}$, $[l,x]=0$ implies $x=\beta l$
for some $\beta \in \mathbf{k}$. \\
\noindent
(b) $C(l)$ is one dimensional for all nonzero $l \in \mathfrak{L}$. \\
\noindent
(c) $\mathfrak{L}$ does not contain any abelian Lie algebras of dimension
greater than one. \\
\noindent
(d) If $\alpha, \beta \in \mathfrak{L}$ are linearly independent, 
then $[\alpha,\beta] \neq 0$.
\end{pro}
\noindent
\begin{proof}
The proof is easy and left to the reader.
\end{proof}

\begin{defn}
A Lie algebra $\mathfrak{L}$ is said to be self-centralizing if 
it satisfies any of the equivalent conditions of 
Proposition~\ref{pro: TFAEsc}.
\end{defn}

\begin{rem}
Thus a self-centralizing Lie algebra is one where the centralizers have
as small a dimension as possible. Notice that a self-centralizing Lie
algebra of dimension strictly greater than one must have trivial center.
Furthermore, a Lie algebra isomorphic to a self-centralizing one, is itself
self-centralizing.
\end{rem}

\begin{rem}
It is easy to check that the nonabelian Lie algebra of dimension two is
self-centralizing but is not simple. Similarly $\mathfrak{sl}_n$, the
Lie algebra of $n \times n$, trace zero matrices is simple but 
contains an abelian Lie subalgebra of dimension greater than one 
 for $n \geq 3$ and hence is not self-centralizing.
\end{rem}

We now make a useful observation:

\begin{thm}
\label{thm: sc}
For any stable algebra $A$, $Witt(A)$ is a self-centralizing
Lie algebra.
\end{thm}
\begin{proof}
Let $f \partial$ be a nonzero element of $Witt(A)$. Suppose
$[f\partial, g\partial]=0$. Then as
$[f\partial, g\partial] = (fg'-gf')\partial$, we conclude
that $fg'-gf'=0$ in $A \subseteq \RR$.
 
Then we can rewrite
$fg'-gf'=0$ as $(g/f)'f^2=0$ in $\RR$ which is possible since $f$ is not the
zero element. Since the only elements in $\RR$ which have
zero derivative, are the constants, we conclude that $g/f$ is a constant
or that $g$ is a multiple of $f$. Thus we conclude $C(f\partial)$ is one 
dimensional.
This concludes the proof.
\end{proof}

\begin{rem}
It follows immediately from Theorem~\ref{thm: sc}, that the classical
Witt algebra and the centerless Virasoro algebra are self-centralizing.
\end{rem}

\begin{defn}
Recall that a 
Lie algebra is called semisimple if it does not possess any nontrivial 
solvable ideals. It is a standard fact that a Lie algebra is semisimple if 
it does not possess any nontrivial abelian ideals. (See~\cite{Jacobson}.)   
\end{defn}

Let us record some consequences of the self-centralizing property in the
following proposition.

\begin{pro}
\label{pro: sc}
Let $\mathfrak{L}$ be a self-centralizing Lie algebra, then:

\noindent
(a) Any Lie subalgebra is also self-centralizing. \\
\noindent
(b) If $\mathfrak{L}$ possesses a finite dimensional ideal $I$ of 
dimension $n>1$,
then $\dime(\mathfrak{L}) \leq n^2$. If $\mathfrak{L}$ possesses an ideal 
of dimension 1,
then $\dime(\mathfrak{L}) \leq 2$. \\
\noindent
(c) If $\mathfrak{L}$ is infinite dimensional, then $\mathfrak{L}$
does not possess any finite dimensional, nontrivial ideals. \\
\noindent
(d) If $\alpha, \beta$ are two linearly independent elements of $\mathfrak{L}$
and $x$ is a common eigenvector of $ad(\alpha)$ and $ad(\beta)$
then $x$ is a multiple of $[\alpha, \beta]$. \\ 
\noindent
(e) If $\alpha, \beta$ are two linearly independent elements of
$\mathfrak{L}$, then there is no basis for $\mathfrak{L}$, in which
both $\alpha$ and $\beta$ are ad-diagonal. \\ 
\noindent
(f) $\mathfrak{L}$ is indecomposable i.e., $\mathfrak{L}$ cannot
be written as a direct sum of two nonzero Lie algebras. \\
(g) If $\dime(\mathfrak{L})>2$ then $\mathfrak{L}$ is semisimple. \\
\noindent
(h) If $\mathfrak{L}$ is finite dimensional and $\mathbf{k}$ is algebraically
closed, then $\mathfrak{L}$ is either isomorphic to the nonabelian Lie algebra
of dimension two, $\mathfrak{sl}_2$, or a Lie algebra of dimension less than
or equal to one. 
\end{pro}
\begin{proof}
(a) follows at once from the definition of a self-centralizing Lie algebra.
To prove (b), suppose $I$ is a nontrivial, finite dimensional ideal of
dimension $n$. Then
define
$\theta: \mathfrak{L} \rightarrow \Endo_{\mathbf{k}}(I)$ by
$$
\theta(x) = ad(x) |_I.
$$
Note that $\Endo_{\mathbf{k}}(I)$
is finite dimensional of dimension $n^2$. If $n>1$, then  
$\theta$ is injective by the self-centralizing property of $\mathfrak{L}$.
This is because if $z$ were a nonzero element in $\kerl (\theta)$, then
$I \subseteq C(z)$. However, $C(z)$ has dimension 1 as $\mathfrak{L}$
is self-centralizing, while $I$ is assumed to have dimension bigger
than 1 giving a contradiction.
It follows easily from the injectivity of $\theta$ that 
$$
\dime(\mathfrak{L}) \leq \dime(\Endo_{\mathbf{k}}(I))=n^2.
$$
If $n=1$ and $x$ is a generator of $I$, then $\kerl (\theta)$ is codimension
at most one in $\mathfrak{L}$. However, $\kerl (\theta)=C(x)=I$ since
$\mathfrak{L}$ is self-centralizing. Thus $\dime(\mathfrak{L}) \leq 2$.

(c) follows quickly from (b). (d) and (e) follow from quick 
calculations and the self-centralizing property. (f) is a trivial verification.

For (g), note that if $\dime(\mathfrak{L}) > 2$, then by (b), $\mathfrak{L}$
does not possess any nontrivial ideals of dimension one. On the other hand,
because $\mathfrak{L}$ is self-centralizing, it cannot possess any
 abelian ideals of dimension greater than one and so we conclude that
$\mathfrak{L}$ does not possess any nontrivial abelian ideals and hence
is semisimple.

For (h), note that if $\dime(\mathfrak{L}) \leq 2$, the result is easy.
So we can assume $2 < \dime(\mathfrak{L}) < \infty$, and so by (g),
$\mathfrak{L}$ is semisimple. From standard results (see~\cite{Jacobson} 
or \cite{Hum}),
since we are over a field of characteristic zero, $\mathfrak{L}$
is the direct sum of simple Lie algebras. However by (f), we see
that in fact $\mathfrak{L}$ must be simple. 

If we assume $\mathbf{k}$ to be algebraically closed, 
then the Cartan subalgebra of $\mathfrak{L}$ is abelian, and since 
$\mathfrak{L}$ is self-centralizing, it must have rank one. 
From the classification
of simple finite dimensional Lie algebras over an algebraically closed
field, we see that $\mathfrak{L}$ is isomorphic to $\mathfrak{sl}_2$.
\end{proof}

\begin{rem}
By Proposition~\ref{pro: sc}, 
we see that there aren't very many finite dimensional
self-centralizing Lie algebras. Thus it is somewhat striking that
all of the generalized Witt algebras are self-centralizing. 

We will see later that we can find infinitely many nonisomorphic
generalized Witt algebras so that the class of self-centralizing
Lie algebras is pretty rich. In the class of infinite dimensional
Lie algebras, Proposition~\ref{pro: sc} shows that being self-centralizing
is a stronger condition than being semisimple and yet is usually easier
to verify than simplicity. 
\end{rem}

Since stable algebras $A$ are infinite dimensional in all but some
trivial cases, $Witt(A)$ is usually infinite dimensional and since
it is self-centralizing by Theorem~\ref{thm: sc}, it follows 
by Proposition~\ref{pro: sc}, 
that $Witt(A)$ is both semisimple and indecomposable.
However there are examples where $Witt(A)$ is simple and there are examples
where it is not. We will discuss this more later on.

\section{Eigenvalues and eigenspaces}

We have seen that all generalized Witt algebras are self-centralizing.
Given a Lie algebra $\mathfrak{L}$, and $\alpha \in \mathfrak{L}$,
let $E_a(\alpha) \subseteq \mathfrak{L}$ be the eigenspace of $ad(\alpha)$
corresponding to the eigenvalue $a \in \mathbf{k}$. 

In this language, a self-centralizing Lie algebra $\mathfrak{L}$ 
is one such that
$$
\dime(E_0(\alpha))=1
$$ 
for all nonzero $\alpha \in \mathfrak{L}$.
We have seen that a generalized Witt algebra is self-centralizing
and hence satisfies this condition on the eigenspaces. We will
now extend this result by studying further constraints on these eigenspaces
in a generalized Witt algebra.

Before we can do this, we need to recall the concept of the logarithmic
derivative on $\RR$, and some of its basic properties.

\begin{defn}
Let $\mathbf{R}^\sharp$ denote the group of nonzero elements in the field $\RR$
under multiplication. (Recall $\RR$ is the field of fractions of
$\mathbf{k}[[x]]$.)
The logarithmic derivative $LD: \mathbf{R}^\sharp 
\rightarrow \RR$ is defined by
$$
LD(f) = \frac{f'}{f}.
$$
where $f'$ is the formal derivative of $f$. It is easy to check that
$LD$ is a group homomorphism from $(\mathbf{R}^\sharp, \times)$ to $(\RR,+)$.

It is also routine to see that $\kerl (LD)$ is exactly the constant functions.
Thus if $u,v \in \mathbf{R}^\sharp$ have $LD(u)=LD(v)$, then $u$ is a scalar multiple
of $v$.
\end{defn}

Now we are ready to prove an important lemma which generalizes
Theorem~\ref{thm: sc}.

\begin{lem}
\label{lem: eigen1}
If $f\partial \in Witt(\RR)$ is a nonzero element, then
$\dime(E_a(f\partial)) \leq 1$ for all $a \in \mathbf{k}$.
Furthermore, if $g\partial$ is a nonzero element in $E_a(f\partial)$,
then $g=fu$ where $LD(u) = a/f$.
\end{lem}

\begin{proof}
Suppose $g\partial$ is a nonzero
element in $E_a(f\partial)$. Then
\begin{align*}
\begin{split}
[f\partial,g\partial] &= a g\partial \\
(fg'-gf') \partial &= a g \partial \\
 (g/f)'f^2 &= ag
\end{split}
\end{align*}

Thus we conclude $ (g/f)'f = a(g/f) $. If we let $u=g/f$, this becomes
$ u'f = au $ or $LD(u) = a/f$. Thus we conclude $g=fu$ where $LD(u) = a/f$.
If $h\partial$ is another nonzero element in $E_a(f\partial)$, then
similarly we would conclude $h=fv$ where $LD(v) = a/f$.
However $LD(u)=LD(v)=a/f$ so $v$ is a scalar multiple of $u$ and hence
$h$ is a scalar multiple of $g$. Thus we see $\dime(E_a(f\partial)) \leq 1$
as we sought to show.
\end{proof}

Lemma~\ref{lem: eigen1} shows that for any nonzero $f\partial \in Witt(\RR)$,
and $a \in \mathbf{k}$, the eigenspace of $ad(f\partial)$ corresponding
to $a$ is at most one dimensional. It remains to decide when this eigenspace
is one dimensional and when it is zero dimensional.
To do this, it turns out we need to find the image of 
$LD: \RR^\sharp \rightarrow
\RR$. We will now introduce a few more concepts in formal calculus
 that will let us do this.

\begin{defn}
Given a nonzero $f \in \RR$, we can write
$$
f=\sum_{i=N}^\infty \alpha_i x^i
$$
where $\alpha_i \in \mathbf{k}$ for all $i \geq N$ and $\alpha_N \neq 0$.
$N$ is called the Weierstrass degree (see~\cite{Lang}) of $f$ and will
be denoted by $W(f)$. $\alpha_{-1}$ is called the residue of $f$ and will
be denoted $\res (f)$. We also define $W(0)=\infty$ and $\res (0)=0$.
\end{defn}

\begin{defn}
Let $U = \{ f \in \RR | W(f)=0 \}$. Then $f \in U$ if and only if
$f \in \mathbf{k}[[x]]$ and $f(0) \neq 0$ and this happens 
if and only if $f$ is a unit
of $\mathbf{k}[[x]]$. Thus $U$ is the group of units of $\mathbf{k}[[x]]$
under multiplication.
\end{defn}

We now collect some elementary properties of the Weierstrass degree
in the next lemma. The proof is simple and will be left to the reader.

\begin{lem}
Given nonzero $f \in \RR$, we can write 
$$
f = x^{W(f)}u
$$
with $u \in U$. Furthermore such an expression for $f$ is unique. \\
\noindent
Given $f,g \in \RR$,
$$
W(fg) = W(f) + W(g).
$$
\end{lem}

We now define formal integration:

\begin{defn}
Recall $(x)$ is the unique maximal ideal of $\mathbf{k}[[x]]$.
We define formal integration
$\int : \mathbf{k}[[x]] \rightarrow (x)$ by
\begin{align*}
\begin{split}
\int(\sum_{i=0}^\infty \alpha_i x^i) &= 
\sum_{i=0}^\infty \alpha_i \frac{x^{i+1}}{i+1} \\
&= \sum_{i=1}^\infty \alpha_{i-1} \frac{x^i}{i}
\end{split}
\end{align*}
It follows easily that $\int \in \Endo_{\mathbf{k}}(\mathbf{k}[[x]])$
and that if $f \in \mathbf{k}[[x]]$ is nonzero, 
$$
W(\int f) = W(f) + 1.
$$
Furthermore, we have of course 
$$
\partial (\int f) = f
$$ 
for all $f \in 
\mathbf{k}[[x]]$.
\end{defn}  

We will also need to compose two power series. Recall that 
given $g \in \mathbf{k}[[x]]$ and $f \in (x)$, we 
have a well-defined 
composition power series $g \circ f \in \mathbf{k}[[x]]$ given in the
following manner:
If $g = \sum_{i=0}^\infty \alpha_i x^i$ then 
$g \circ f \in \mathbf{k}[[x]]$ is given formally by 
$\sum_{i=0}^\infty \alpha_i f^i$.

We collect well-known results on this composition in the following proposition:

\begin{pro}
\label{pro: lastcompose}
If $g \in \mathbf{k}[[x]]$ and $f \in (x)$. Then there exists
a series $g \circ f \in \mathbf{k}[[x]]$ such that
$$
(g \circ f)' = (g' \circ f)f'.
$$
Furthermore, $(g\circ f)(0) = g(0)$ and $g \circ x = g$.
\end{pro}

We are now ready to study the image of the logarithmic derivative
 $LD: \RR^\sharp \rightarrow \RR$.

\begin{lem}
\label{lem: notimage}
Let $u \in \RR^\sharp$, \\
\noindent
(a) If $W(u) \neq 0$ then $W(LD(u))=-1$ and $\res (LD(u))$ is equal
to $W(u)$ which is of course an integer. \\
\noindent
(b) If $W(u) = 0$ then $W(LD(u)) \geq 0$. \\
\noindent
(c) If $W(g) < -1$ or if $W(g) = -1$ and $\res (g)$ is not an integer,
 then $g$ is not in the image of $LD: \RR^\sharp \rightarrow
\RR$.
\end{lem}
\begin{proof}
The proof will be left to the reader. It 
follows from writing $u$ as a Laurent series and explicitly 
calculating $LD(u)$.
\end{proof}

We have seen in Lemma~\ref{lem: notimage}, conditions that ensure
an element $g \in \RR$ is not in the image of $LD: \RR^\sharp \rightarrow
\RR$. We now show that in the remaining situations, the element $g$ is in
the image.

First recall $e^x \in \mathbf{k}[[x]]$ is the power series given
by
$$
e^x = \sum_{i=0}^\infty \frac{x^i}{i!}.
$$
It is easy to verify that $\partial e^x = e^x$ and that 
$e^x$ evaluated at $x=0$ is $1$.

Given $g \in \mathbf{k}[[x]]$, $\int g$ lies in $(x)$, the maximal
ideal of $\mathbf{k}[[x]]$. Thus by Proposition~\ref{pro: lastcompose}
we can form the power series $e^x \circ (\int g)$ which we will denote
by $e^{\int g}$. It follows from the same proposition that
$$
\partial e^{\int g} = e^{\int g} \partial( \int g) = ge^{\int g}.
$$
Furthermore since $e^{\int g}(0)=e^x(0)=1$, we see that $e^{\int g} \in U$
for all $g \in \mathbf{k}[[x]]$.

We will use these facts in the next theorem.

\begin{thm}
\label{thm: imageLD}
Let $g \in \RR$. Then either: \\
\noindent
(a) $W(g) \geq 0$ and $g=LD(e^{\int g})$. \\
\noindent
(b) $W(g) = -1$ and $\res (g)$ is an integer
then $g=\frac{\res (g)}{x} + u$ for some unique $u \in \mathbf{k}[[x]]$
and we have $g=LD(x^{\res (g)}e^{\int u})$. \\
\noindent
(c) $W(g) < -1$ or $W(g) = -1$ and $\res (g)$ is not an integer
in which case $g$ is not in the image of $LD: \RR^\sharp \rightarrow \RR$.
\end{thm}
\begin{proof}
(c) follows from Lemma~\ref{lem: notimage}. For (a), assume $g$ has
$W(g) \geq 0$ so that $e^{\int g} \in U$. Then we calculate
$$
LD(e^{\int g}) = \frac{\partial e^{\int g}}{e^{\int g}}
= \frac{ g e^{\int g}}{e^{\int g}} = g
$$
and so (a) is proven. 

Assume $g$ as in the statement of (b). Then it
is obvious that we may write
$g= \frac{\res (g)}{x} + u$ with $u \in \mathbf{k}[[x]]$ determined uniquely.
Since $\res (g)$ is an integer $x^{\res (g)}e^{\int u}$ certainly defines
an element in $\RR^\sharp$. We compute
\begin{align*}
\begin{split}
LD(x^{\res (g)}e^{\int u}) &= \res (g)LD(x) + LD(e^{\int u}), \text{  since }
 LD \text{ is a homomorphism} \\
&= \res (g)\frac{1}{x} + u, \text{    using the calculation in (a) } \\
&= g.
\end{split}
\end{align*}
Thus we are done.

\end{proof}

We are now ready to complete the analysis of the eigenspaces of 
elements in $ad(Witt(\RR))$
which was started in Lemma~\ref{lem: eigen1}.

\begin{thm}[Spectral theorem for $\RR$]
\label{thm: specRR}
Let $f\partial$ be a nonzero element in $Witt(\RR)$. Then: \\
\noindent
(a) 
If $W(f)>1$, then $\dime(E_a(f\partial))=0$ for all nonzero $a \in \mathbf{k}$
and  
$$
\dime(E_0(f\partial))=1.
$$

\noindent
(b) 
If $W(f) \leq 0$, then $\dime(E_a(f\partial))=1$ for all $a \in \mathbf{k}$.
Furthermore, 
$$
fe^{\int \frac{a}{f}}\partial \in E_a(f\partial).
$$

\noindent
(c) 
If $W(f)=1$ then $\dime(E_a(f\partial))=0$ if $a \neq Nf'(0)$ for some
 integer $N$. $\dime(E_{Nf'(0)}(f\partial))=1$ for all $N \in \mathbb{Z}$.
Furthermore 
$$
fx^Ne^{\int (\frac{N(f'(0)x-f)}{fx})}\partial \in E_{Nf'(0)}(f\partial)
$$ 
for all $N \in \mathbb{Z}$.
\end{thm}
\begin{proof}
Let $f\partial \in Witt(\RR)$ be nonzero and let $a \in \mathbf{k}$.
Then by Lemma~\ref{lem: eigen1}, we see that 
$\dime(E_a(f\partial))$ is either zero or one and it is one if and only
if $\frac{a}{f} = LD(u)$ for some $u \in \RR^\sharp$. Furthermore,
in this case, $fu\partial$ is a nonzero element of $E_a(f\partial)$.
Since we know $\dime(E_0(f\partial))=1$ we can assume $a \neq 0$ for the
rest of the proof. It follows that $W(\frac{a}{f})=-W(f)$.

If $W(f)>1$ then $W(\frac{a}{f})<-1$ and so by
Theorem~\ref{thm: imageLD}, $\frac{a}{f}$ is not in the image of the
logarithmic derivative and hence we have proven (a).

If $W(f) \leq 0$ then $W(\frac{a}{f})\geq 0$ and so
$\frac{a}{f} = LD(e^{\int \frac{a}{f}})$ by Theorem~\ref{thm: imageLD}
giving us (b).

If $W(f)=1$ then we can write $f=xf'(0)v$ where $v \in U$ has $v(0)=1$.
Then $W(\frac{a}{f})=-1$ and $\res (\frac{a}{f})=\frac{a}{f'(0)}$.
Again by Theorem~\ref{thm: imageLD}, $\frac{a}{f}$ is in the image of
the logarithmic derivative if and only if this residue is an integer
which happens if and only if $a$ is an integral multiple of $f'(0)$.
If this is the case, then $a=Nf'(0)$ and we can write 
$$
\frac{Nf'(0)}{f} = \frac{N}{x} + w
$$
where $w \in \mathbf{k}[[x]]$. Theorem~\ref{thm: imageLD} then shows
that $\frac{Nf'(0)}{f} = LD(x^Ne^{\int w}).$
Now it remains only to note that
$$
w = \frac{Nf'(0)}{f} - \frac{N}{x} = \frac{N(f'(0)x-f)}{fx}
$$
and we are done.
\end{proof}

\section{Spectra}

We now discuss the concept of a spectrum which we will find to be
very useful in the remainder of this paper.

\begin{defn}
Given a Lie algebra $\mathfrak{L}$, and $\alpha \in \mathfrak{L}$,
we define the $\mathfrak{L}$-spectrum of $\alpha$ to be
$$
\spec_{\mathfrak{L}}(\alpha) = \{ a \in \mathbf{k} | \dime(E_a(\alpha)) \neq 0 \}.
$$
We write $\spec(\alpha)$ for $\spec_{\mathfrak{L}}(\alpha)$ when there
is no danger of confusion. Thus the spectrum of $\alpha$ is the
set of eigenvalues of $ad(\alpha) \in \Endo_{\mathbf{k}}(\mathfrak{L})$.
\end{defn}

Notice that in a nonzero Lie algebra $\mathfrak{L}$, $\spec(0)=\{ 0 \}$ and
$0 \in \spec(\alpha)$ for all $\alpha \in \mathfrak{L}$. In general
the spectrum possesses no significant algebraic structure. However, we
will soon see that if $\mathfrak{L}$ is self-centralizing, $\spec(\alpha)$
possesses the structure of a pseudomonoid (which we will define shortly)
for all $\alpha \in \mathfrak{L}$.

\begin{defn}
A subset $P$ of $\mathbf{k}$ is a pseudomonoid if it satisfies 
the following conditions: \\
\noindent
(a) $0 \in P$. \\
\noindent
(b) If $a, b \in P$ and $a \neq b$ then $a + b \in P$ where $+$ is addition
in $\mathbf{k}$.
\end{defn}

\begin{rem}
Notice that a pseudomonoid differs from a monoid because in a monoid
we may also add an element to itself, i.e., if $a \in P$ and $P$ is a monoid
under $+$ then $a + a \in P$. This need not hold for a pseudomonoid as can be
seen by the following example: \\
Let $A=\{-1,0,1,\dots \}$ be the set of integers greater than 
or equal to negative
one. This set is a pseudomonoid under addition but is not a monoid as
$$
(-1) + (-1) = -2 \notin A.
$$
\end{rem}

The concept of a pseudomonoid turns out to be important for us because
of the following lemma:

\begin{lem}
\label{lem: specscpm}
Let $\mathfrak{L}$ be a Lie algebra, and $\alpha \in \mathfrak{L}$ be
a nonzero element. Then for any $a, b \in \mathbf{k}$, we have 
$[E_a(\alpha), E_b(\alpha)] \subseteq E_{a+b}(\alpha)$. 

Thus if $\mathfrak{L}$ is self-centralizing, then $\spec(\alpha)$ is a
pseudomonoid for all $\alpha \in \mathfrak{L}$.
\end{lem}
\begin{proof}
For a proof of the first statement, take $\alpha \in \mathfrak{L}$ and 
$a, b \in \mathbf{k}$. Then for $e_a \in E_a(\alpha)$ and 
$e_b \in E_b(\alpha)$ we have by the Jacobi identity:
\begin{align*}
\begin{split}
[\alpha, [e_a,e_b]] &= [[\alpha,e_a],e_b] + [e_a, [\alpha,e_b]] \\
&= [ae_a,e_b] + [e_a,be_b] \\ 
&= (a+b)[e_a,e_b].
\end{split}
\end{align*}
Thus we see that $[e_a,e_b] \in E_{a+b}(\alpha)$ which proves the first
statement. 

Now suppose that $\mathfrak{L}$ is self-centralizing. $\spec(0)=\{0\}$ is a
pseudomonoid so assume $\alpha \neq 0$. Let $a, b \in \spec(\alpha)$ 
with $a \neq b$. Then if we take nonzero $e_a \in E_a(\alpha)$ and $e_b \in 
E_b(\alpha)$, since $a \neq b$ it follows that $e_a, e_b$ are linearly
independent. Thus since $\mathfrak{L}$ is self-centralizing, it follows
that $[e_a,e_b]\neq 0$ which shows
that $E_{a+b}(\alpha) \neq 0$. Thus $a+b \in \spec(\alpha)$ and so
$\spec(\alpha)$ is a pseudomonoid.  
\end{proof}

\begin{defn}
Let $\alpha \in \mathfrak{L}$. Then we define 
$$
M_{\mathfrak{L}}(\alpha) = \oplus_{a \in \mathbf{k}} E_a(\alpha).
$$
Thus $M_{\mathfrak{L}}(\alpha)$ is the subspace of $\mathfrak{L}$ spanned
by the eigenspaces of $\alpha$. It is the maximal subspace on which 
$ad(\alpha)$ is diagonal (with respect to some basis). 
\end{defn}

It is easy to argue that we can also write
$$
M_{\mathfrak{L}}(\alpha) = \oplus_{a \in \spec(\alpha)} E_a(\alpha).
$$
It follows from Lemma~\ref{lem: specscpm} that $M_{\mathfrak{L}}(\alpha)$ 
is a Lie subalgebra of $\mathfrak{L}$. 
We will write $M(\alpha)$ for $M_{\mathfrak{L}}(\alpha)$ when there is 
no danger of confusion.

We will now look at a few examples before proceeding any further. To do this,  
it is useful to introduce the concept of a differential spanning set.

\begin{defn}
$S \subseteq \RR$ is called a differential spanning set if it satisfies
the following conditions: \\
\noindent
(a) $1 \in S$. \\
\noindent
(b) If $f,g \in S$, then $fg \in S$. \\
\noindent
(c) If $f \in S$ then $\partial f$ is a linear combination of elements in $S$.

Given a differential spanning set $S$, the vector space $A$ spanned by $S$ 
in $\RR$ is easily seen to be a stable algebra.
\end{defn}

\begin{ex}
\label{ex: one}
Let $S=\{x^n | n \in \mathbb{N} \}$, then it is easy to check that $S$ is
a differential spanning set which spans the polynomial stable algebra
$\mathbf{k}[x]$ and is in fact a basis for this algebra. In 
$Witt(\mathbf{k}[x])$, one calculates:
\begin{align*}
\begin{split} 
[x^n \partial, x^m \partial] &= (x^n(x^m)' - x^m(x^n)')\partial \\
&= (m-n)x^{m+n-1}\partial.
\end{split}
\end{align*}
Thus we see easily that $M(x\partial)=Witt(\mathbf{k}[x])$ and that
$\spec(x\partial)=\{-1,0,1,\dots \}$.
\end{ex}

\begin{ex}
\label{ex: two}
Let $S=\{x^n | n \in \mathbb{Z} \}$, then $S$ is a differential spanning set 
which forms a basis for the Laurent polynomial stable algebra 
$\mathbf{k}[x,x^{-1}]$. Exactly as in Example~\ref{ex: one}, one can show 
that $M(x\partial)=Witt(\mathbf{k}[x,x^{-1}])$ and that 
$\spec(x\partial)=\{ \dots,-2,-1,0,1,2,\dots \} =\mathbb{Z}$.
\end{ex}

Note that the spectrum of $x\partial$ depends on which Lie algebra we are in 
and so we stress that the reader should keep in mind the surpressed subscript
$\mathfrak{L}$ in the notation for $\spec$.

\begin{ex}
\label{ex: three}
Let $G$ be a submonoid of $(\mathbf{k},+)$, then 
$S=\{ e^{ax} | a \in G \}$, is a differential spanning set (since $e^{(a+b)x}=
e^{ax}e^{bx}$ as the reader can verify). Let $A(G)$ be the stable 
algebra that this spanning set spans. In $Witt(A(G))$, we calculate:
\begin{align*}
\begin{split}
[e^{ax}\partial,e^{bx}\partial] &= (e^{ax}be^{bx} - e^{bx}ae^{ax})\partial \\
&=(b-a)e^{(a+b)x}.
\end{split}
\end{align*}
>From this, it follows that $M(1\partial)=Witt(A(G))$ and that 
$\spec(1\partial)=G$. Since $e^{bx}\partial 
\in E_b(1\partial)$ for all $b \in G$, 
it also follows that $S$ is a basis for $A(G)$.
\end{ex}

\begin{rem}
It is a standard fact that every 
torsion-free abelian group embeds into a torsion-free 
divisible group and that a torsion-free divisible group is 
isomorphic to the additive group of a rational vector space (see~\cite{Rob}). 

Any rational vector space is isomorphic to a subgroup of $(\mathbf{k},+)$ 
for suitable choice of $\mathbf{k}$. (Need the dimension of $\mathbf{k}$ over 
its characteristic subfield $\mathbb{Q}$ to be big enough.)

Thus by Example~\ref{ex: three}, we conclude that every torsion-free 
abelian group is the spectrum of some ad-diagonal element in some 
generalized Witt algebra in one 
variable.
\end{rem}

Finally we state a general spectral theorem for generalized Witt algebras.
It is based on Theorem~\ref{thm: specRR}.

\begin{thm}[Spectral theorem]
\label{thm: spectraltheorem} 
Let $Witt(A)$ be a generalized Witt algebra and $f\partial$ be a nonzero
 element in $Witt(A)$. Then for all $a \in \mathbf{k}$, 
$$
\dime(E_a(f\partial)) \leq 1
$$ 
and: \\
\noindent
(a) If $W(f)>1$, then $\spec(f\partial)=\{0\}$. \\
\noindent
(b) If $W(f)\leq 0$ then for all $a \in \mathbf{k}$, 
$$
a \in \spec(f\partial) \iff fe^{\int \frac{a}{f}} \in A.
$$
\\
\noindent
(c) If $W(f)=1$ then $\spec(f\partial) \subseteq \mathbb{Z}f'(0)$, where
$\mathbb{Z}f'(0)$ stands for the set of integral multiples 
of $f'(0) \in \mathbf{k}$.
Furthermore, for all $N \in \mathbb{Z}$,
$$
Nf'(0) \in \spec(f\partial) \iff fx^Ne^{\int \frac{N(f'(0)x-f)}{fx}} \in A.
$$
\end{thm}
\begin{proof}
First note that $A \subseteq \RR$ so $Witt(A)$ is a Lie subalgebra of 
$Witt(\RR)$. Then if $f\partial \in Witt(A)$, and $a \in \mathbf{k}$, 
the $a$-eigenspace of $ad(f\partial)$ for $Witt(A)$ lies inside the one 
for $Witt(\RR)$. Thus $\spec_{Witt(A)}(f\partial) \subseteq 
\spec_{Witt(\RR)}(f\partial)$ and $a \in \spec_{Witt(\RR)}(f\partial)$ lies
in $\spec_{Witt(A)}(f\partial)$ if and only if one of the eigenvectors
in $\RR$ corresponding to $a$ actually lies in $A$. 
With these comments, the rest now follows from Theorem~\ref{thm: specRR}.
\end{proof}

Theorem~\ref{thm: spectraltheorem}, will show that our definition 
of generalized Witt algebras is related to the definition 
in papers such as~\cite{Dok}.
We do this in the next proposition.

\begin{pro}
\label{pro: otherdefn}
Let $Witt(A)$ be a generalized Witt algebra and let $f\partial$ be a 
nonzero element of $Witt(A)$. Then there exists a basis 
$\{e_a\}_{a \in \spec(f\partial)}$ of $M(f\partial)$ such that 
$$
[e_a, e_b]=(b-a)e_{a+b}
$$
for all $a,b \in \spec(f\partial)$. (Here $(b-a)e_{a+b}$ is 
considered to be zero for $a=b$ even though $a+b$ might not 
be in $\spec(f\partial)$.)
Furthermore, we can take $e_a \in E_a(f\partial)$ for all 
$a \in \spec(f\partial)$ and $e_0=f\partial$.
\end{pro}
\begin{proof}
By Theorem~\ref{thm: spectraltheorem}, if $W(f) > 1$, then 
$\spec(f\partial)=\{0\}$ and the result is obvious.

If $W(f) \leq 0$, then set
$$
e_a = fe^{\int \frac{a}{f}} \partial
$$
for all $a \in \spec(f\partial)$.

Then by Theorem~\ref{thm: spectraltheorem}, $\{e_a\}_{a \in \spec{f\partial}}$
is a basis for $M(f\partial)$.

One computes using $[g\partial, h\partial]=(gh'-hg')\partial$, 
that indeed 
$$
[e_a,e_b]=(b-a)e_{a+b}.
$$

Similarly, in the remaining case where $W(f)=1$, we set 
$$
e_{Nf'(0)}=fx^Ne^{\int \frac{N(f'(0)x-f)}{fx}} \partial
$$
for all $Nf'(0) \in \spec(f\partial)$, and again compute that 
$$
[e_{Nf'(0)}, e_{Mf'(0)}]=(M-N)f'(0)e_{(N+M)f'(0)}
$$
for all $Mf'(0),Nf'(0) \in \spec(f\partial)$.
\end{proof}

We are now ready to study the issue of simplicity of a generalized Witt 
algebra. We will do this in the next section.

\section{Simplicity}

\begin{defn}
A Lie algebra $\mathfrak{L}$ is said to be strongly graded if there exists a
pseudomonoid $G$ and a vector space decomposition:
$$
\mathfrak{L} = \oplus_{a \in G}E_a
$$
with the following properties: \\
\noindent
(a) $\dime(E_a)=1$ for all $a \in G$. \\
\noindent
(b) There is a basis $\{e_a\}_{a \in G}$ of $\mathfrak{L}$ such 
that $e_a \in E_a$ for all $a \in G$ and 
$$
[e_a,e_b]=(b-a)e_{a+b}.
$$
Note that this means that $\spec(e_0)=G$. 
\end{defn}

\begin{rem}
Of course, by Theorem~\ref{thm: spectraltheorem}  
and Proposition~\ref{pro: otherdefn}, if $Witt(A)$ is a generalized 
Witt algebra, and $\alpha \in Witt(A)$ is nonzero, then 
$M(\alpha)$ is a strongly graded Lie algebra, graded by the pseudomonoid 
$\spec(\alpha)$ where $\alpha$ plays the role of $e_0$.
\end{rem}

\begin{rem}
It is obvious that two strongly graded Lie algebras, graded by the 
same pseudomonoid $G \subseteq \mathbf{k}$, are isomorphic as Lie 
algebras.
\end{rem}

Now we set out to get a complete correspondence between the ideals of a 
strongly graded Lie algebra and the ideals of the pseudomonoid which grades it.

\begin{defn}
Let $G$ be a pseudomonoid. 

Then $S \subseteq G$ is called a closed subset if
for all distinct $a,b \in S$, we have $a+b \in S$. Note that a closed subset 
$S$ need not be a subpseudomonoid of $G$ since we do not require that 
$0 \in S$. In fact the empty set $\emptyset$ is always a closed subset.

$I \subseteq G$ is called an ideal subset if for all $a \in I$ and all 
$b \in G$ such that $b \neq a$, we have $a+b \in I$. Again the empty set 
is always an ideal subset and $G$ is always an ideal subset of $G$.
These are called the trivial ideal subsets. 

A pseudomonoid $G$ which has no nontrivial ideal subsets is called a simple 
psuedomonoid. 

A nonzero element $x \in G$ is called invertible if $-x \in G$. (Recall, all
 pseudomonoids are by definition in $\mathbf{k}$ and hence $-x$ exists in 
$\mathbf{k}$ and is distinct from $x$.)
\end{defn}

Fix a strongly graded Lie algebra $\mathfrak{L}$, 
graded by a pseudomonoid $G$, then $\mathfrak{L}=\oplus_{g \in G} E_g$ 
such that there is $e_0 \in E_0$, with $\spec(e_0)=G$ and $E_g$ equal to 
the eigenspace of 
$ad(e_0)$ corresponding to $g$.
 
Then for any $S \subset G$, we define:
$$
\Theta(S) = \oplus_{a \in S}E_a.
$$
(We use the convention that $\Theta(\emptyset)=0$.)

Thus $\Theta$ is a map from the subsets of $G$ to the subspaces of 
$\mathfrak{L}$ which is obviously injective.

Notice that if $S$ is a closed subset of $G$, then $\Theta(S)$ is a Lie 
subalgebra of $\mathfrak{L}$. (Because $[E_a,E_a]=0$ for all $a \in G$.)
Furthermore, if $0 \in S$, then $e_0 \in \Theta(S)$.

Similarly, if $I$ is an ideal subset of $G$, then $\Theta(I)$ is an ideal 
of $\mathfrak{L}$.  

\begin{pro}
\label{pro: idealcor}
Let $\mathfrak{L}$ be a strongly graded Lie algebra, graded by the 
pseudomonoid $G$. 

The map $\Theta$ defined above takes closed subsets of $G$ to Lie 
subalgebras of $\mathfrak{L}$ and this correspondence is injective.

The map $\Theta$ takes closed subsets of $G$ containing $0$, to Lie 
subalgebras of $\mathfrak{L}$ containing $e_0$ and this correspondence is 
bijective.

The map $\Theta$ takes ideal subsets of $G$ to ideals of $\mathfrak{L}$
and this correspondence is bijective.
\end{pro} 
\begin{proof}
All but the surjectivity of the last two correspondences has been proven.

So assume $J$ is an ideal of $\mathfrak{L}$ (or a Lie subalgebra containing 
$e_0$). First, let us show that there is a
subset $I$ of $G$ such that $\Theta(I)=J$.
We can of course 
assume $J \neq 0$ as $\Theta(\emptyset)=0$.

Define $I \subset G$ as follows. Recall that by the grading, for any
 $x \in \mathfrak{L}$, we can write $x$ uniquely as
$$
x= \sum_{a \in G}x_a
$$
with $x_a \in E_a$ and only finitely many $x_a$ nonzero. We call
 $x_a$ the $a$-th component of $x$.
Then set:
$$
I = \{ a \in G \text{ such that there exists } y \in J \text{ whose } 
 a\text{-th component is nonzero} \}.
$$

It is clear that $J \subseteq \Theta(I)$. 
So it remains only to show $\Theta(I) \subseteq J$. We do this by showing that
$E_a \subseteq J$ for any $a \in I$. This follows immediately from the 
following fact:

{\it Fact: If $y \in J$, then all of the components of $y$ are also in $J$.}

We will prove this fact by induction on $n$, the number of nonzero 
components of $y$.
If $n=0,1$, it follows trivially. So assume $n>1$ and we have proven the fact 
for all smaller $n$. So let $y \in J$ and assume 
we can write
$$
y = \sum_{i=1}^n y_{a_i}
$$
with $y_{a_i} \in E_{a_i}$ nonzero 
and $\{a_i\}_{i=1}^n$ a set of distinct elements 
in $I$. Also without loss of generality, $a_1 \neq 0$. 
Then
$$
[e_0,y] = \sum_{i=1}^na_iy_{a_i}
$$
is in $J$ and so
$$
y - \frac{1}{a_1}[e_0,y] = \sum_{i=2}^n (1-\frac{a_i}{a_1})y_{a_i}
$$
is in $J$. However by induction, it follows that the components of
$y - \frac{1}{a_1}[e_0,y]$ lie in $J$ and hence that $y_{a_i}$ lie in $J$
for all $2 \leq i \leq n$. However $y=y_{a_1} + \sum_{i=2}^n y_{a_i}$, so it
also follows that $y_{a_1}$ is in $J$. Thus by induction, 
we have proven the fact
 and hence that $J = \Theta(I)$.

All that remains, is to show that $I$ is an ideal subset if $J$ is an ideal or
 that $I$ is a closed subset containing zero if $J$ is a Lie subalgebra 
containing $e_0$. We prove only the former, the proof of the latter being 
similar.

If $a \in I$ then by definition, there is $y \in J$ such that 
$y = \sum_{g \in G}y_g$ with $y_g \in E_g$ and $y_a \neq 0$.
If $b \in G$ and $b \neq a$, take nonzero $z_b \in E_b$. Then
$[z_b,y] = \sum_{g \in G}[z_b,y_g] \in J$ as $J$ is an ideal. 
Notice that since our pseudomonoids are
defined to be subpseudomonoids of $(\mathbf{k},+)$, the only term in the 
sum that can lie in $E_{a+b}$ is $[z_b,y_a]$ which is nonzero as $z_b, y_a$
are nonzero and since we are in a strongly graded Lie algebra. 
All the other terms, 
live in other eigenspaces and so we conclude $[z_b,y_a]$ has 
nonzero $(a+b)$-component and hence $a+b \in I$ showing that $I$ is an ideal 
subset of $G$.

\end{proof}

\begin{cor}
\label{cor: simplecor}
If $\mathfrak{L}$ is a strongly graded Lie algebra, graded by a pseudomonoid 
$G$. Then $\mathfrak{L}$ is simple if and only if $G$ is simple.

Let $Witt(A)$ be a generalized Witt algebra and 
$\alpha \in Witt(A)$ be nonzero, then $M(\alpha)$ is a simple Lie 
algebra if and only if $\spec(\alpha)$ is a simple pseudomonoid.
(Note it is 
easy to see that $\spec_{M(\alpha)}(\alpha)=\spec_{Witt(A)}(\alpha)$.)
\end{cor}
\begin{proof}
Follows immediately from previous remarks and Proposition~\ref{pro: idealcor}.
\end{proof}

So we see that it would be useful to have some conditions that ensure 
the simplicity of a pseudomonoid. This is the purpose of the next lemma.

\begin{lem}
\label{lem: simplePM}
Let $G$ be a pseudomonoid. Then: \\
\noindent
(a) If $I$ is an ideal subset, and $0 \in I$ then $I=G$. \\
\noindent
(b) If $I$ is an ideal subset, and there is an invertible element 
$x \in I$ then $I=G$. \\
\noindent
(c) A pseudomonoid which is a group is a simple pseudomonoid.
\end{lem}
\begin{proof}
Let $I$ be an ideal subset with $0 \in I$. Then for any nonzero $a \in G$,
we have $0+a=a \in I$ since $I$ is an ideal subset. Thus $I=G$. This proves 
(a). 

Suppose $I$ contained an invertible element $x$. Then as $x \neq -x$, and 
$-x \in G$, we have 
$x + (-x) = 0 \in I$ as $I$ is an ideal subset.  Thus $I=G$ by (a). So this
proves (b).

If $G$ is an (abelian) group and $I$ a nonempty ideal subset.
Then take $a \in I$. If $a=0$ then $I=G$ by (a) and if $a$ is nonzero then 
$a$ is invertible as $G$ is a group, and so $I=G$ by (b). Thus we conclude 
$G$ is a simple pseudomonoid.
\end{proof}

\begin{cor}
If $\mathfrak{L}$ is a strongly graded Lie algebra, graded by an abelian group 
$A \subseteq \mathbf{k}$, then $\mathfrak{L}$ is simple.
\end{cor}

In Example~\ref{ex: one} we saw that the classical Witt algebra, 
$Witt(\mathbf{k}[x])$ is strongly graded, graded by the pseudomonoid 
$G=\{-1,0,1,\dots\}$. If $I$ is a nonempty ideal subset of this pseudomonoid, 
by adding $-1$ repeatedly to an element in $I$ if necessary, we see $-1 \in I$.
Since $-1$ is invertible in $G$, we conclude by Lemma~\ref{lem: simplePM}, 
that $I=G$. So $G$ is a simple pseudomonoid and so the classical Witt algebra 
is a simple Lie algebra.

In Example~\ref{ex: two} we saw that the centerless Virasoro algebra,
$Witt(\mathbf{k}[x,x^{-1}])$ is strongly graded, graded by the pseudomonoid 
$\mathbb{Z}$. Since this is a group, it is simple as a pseudomonoid and we
have proven:

\begin{cor}
The classical Witt algebra and the centerless Virasoro algebra are simple.
\end{cor}

\begin{ex}
\label{ex: notsimple}
The natural numbers $\mathbb{N}=\{0,1,2,\dots\}$ is a monoid which is not 
simple as a pseudomonoid. In fact if we define $I_k = \{k,k+1,\dots \}$ for
 all $k \in \mathbb{N}$, then the reader can easily verify that $I_k$ is an 
ideal subset of $\mathbb{N}$. (There is exactly one more nonempty ideal 
subset not covered by these which we leave the reader to find if they wish.) 
So from Example~\ref{ex: three}, 
$Witt(A(\mathbb{N}))$ gives us an example of a generalized Witt algebra which
is not simple.
\end{ex}
 
\begin{defn}
Two subsets $S_1, S_2$ of $\mathbf{k}$ are said to be equivalent if 
there exists nonzero $k \in \mathbf{k}$ such that 
$$
S_1 = kS_2 \equiv \{ kx | x \in S_2 \}.
$$ 
It is easy to see that this defines an equivalence relation on the subsets of
$\mathbf{k}$. We write $[[S]]$ for the equivalence class of the set $S$ under 
this equivalence relation.
\end{defn}

For any Lie algebra $\mathfrak{L}$, nonzero $\alpha \in \mathfrak{L}$, and 
nonzero $k \in \mathbf{k}$, 
it is easy to see that $M(\alpha)=M(k\alpha)$ and 
$\spec(k\alpha)=k\spec(\alpha)$. Thus we have 
$$
[[\spec(k\alpha)]]=[[\spec(\alpha)]].
$$

It is also easy to see that 
two strongly graded Lie algebras, graded by equivalent pseudomonoids, 
are isomorphic as Lie algebras.

Given a strongly graded Lie algebra $\mathfrak{L}$, graded by the 
pseudomonoid $G$, we have 
$\mathfrak{L}=M(e_0)$ with $\spec(e_0)=G \subseteq \mathbf{k}$ 
where $e_0$ is obtained from the 
definition of a strongly graded Lie algebra. 

We would like to define $[[\spec(e_0)]]$ as an invariant of $\mathfrak{L}$. 
However, it turns out that this is not apriori, intrinsic enough to 
be useful, i.e., it is not obvious that we might not find another 
nonzero element $f$
such that $\mathfrak{L}=M(f)$ and $[[\spec(f)]] \neq [[\spec(e_0)]]$.

In the next section, we show that this in fact cannot occur, and hence define 
an invariant which helps us find infinite families of nonisomorphic 
generalized Witt algebras!

\section{Invariance of the spectrum}

Before we proceed any further, we need to develop a somewhat technical tool.
We need to weakly order any field (of characteristic zero). 
We define this notion now.

\begin{defn}
A weak order on $\mathbf{k}$ is a linear order $\preceq$ on $\mathbf{k}$ 
such that if $x \preceq y$ then $x+z \preceq y+z$ for all $z \in \mathbf{k}$.
(Recall a linear order is a partial order with the property that for any two 
elements $e,f$ either $e \preceq f$ or $f \preceq e$ (or both).)
\end{defn}

Note the field of real numbers $\mathbb{R}$ has a weak order (the usual one) 
and so any subfield of $\mathbb{R}$ has a weak order.

A weak order on an abelian group is defined in exactly the same way.

As is common, we will write $x \prec y$ if $x \preceq y$ and $x \neq y$.

There is also a stronger notion of ordered field in the literature 
(see~\cite{Lang}). 
However for example $\mathbb{C}$, the
 field of complex numbers, cannot be made into an ordered field. 
However, we show in the next 
proposition, that any field (of characteristic zero) has a weak order. 

\begin{pro}
\label{pro: weakorder}
Any field $\mathbf{k}$ (of characteristic zero) possesses a weak order.
\end{pro} 
\begin{proof}
We identify the characteristic subfield of $\mathbf{k}$ with the rational 
numbers $\mathbb{Q}$ as is usual. Then of course, $k$ is a vector space 
over $\mathbb{Q}$.
Define the set $S$ as follows:
$$
S=\{ (A, \preceq) | A \text{ is a } \mathbb{Q}\text{-subspace of } \mathbf{k} 
\text{ and } 
\preceq \text{ is a weak order on } A.\}.
$$
We make $S$ into a partially ordered set $(S, \leq)$ as follows:
$$
(A_1, \preceq_1) \leq (A_2,\preceq_2) \iff  
A_1 \subseteq A_2 \text{ and } \preceq_2 |_{A_1} = \preceq_1.
$$

The characteristic subfield $\mathbb{Q}$ of $\mathbf{k}$ can be viewed as 
the characteristic subfield of the real numbers and so we can put the 
standard order on it. Thus $S$ 
is not empty.

It is easy to verify that any chain $\{(A_i,\preceq_i)_{i \in I}\}$ 
in $(S, \leq)$ has an upper bound $(\cup_{i\in I}A_i,\preceq)$ and thus
Zorn's lemma gives us a maximal element $(M, \preceq)$ of $(S, \leq)$.

Suppose $M \neq \mathbf{k}$, then we can find $a \in \mathbf{k} \setminus M$ 
and thus 
$M'= M \oplus \mathbb{Q}a$ is a $\mathbb{Q}$-subspace of $\mathbf{k}$.  
We define $\preceq'$ on $M'$ as follows:
$$
m_1 + q_1a \prec' m_2 + q_2a \iff m_1 \prec m_2 
\text{ or } m_1=m_2 \text{ and } q_1 < q_2.
$$
It is easy to verify that $\preceq'$ is a weak order on $M'$ which restricts to
$\preceq$ on $M$. 

Thus $(M,\preceq) < (M',\preceq')$ which is a contradiction 
as $(M,\preceq)$ is maximal. Thus we conclude $M=\mathbf{k}$ and hence 
that we can weakly order $\mathbf{k}$. 
\end{proof}

We now use Proposition~\ref{pro: weakorder} to weakly order any pseudomonoid.

\begin{defn}
Let $G \subseteq \mathbf{k}$ be a psuedomonoid. A weak order on $G$ is the 
restriction of some weak order on $\mathbf{k}$. 
\end{defn}

Proposition~\ref{pro: weakorder} shows all pseudomonoids possess a weak order.
(Since we require our pseudomonoids to be in $\mathbf{k}$, by definition.)

\begin{defn}
Let $G$ be a pseudomonoid with weak order $\preceq$. 

We say that $x \in (G, \preceq)$ is positive if $0 \prec x$ and we say 
$x$ is negative if $x \prec 0$.

If we set $P$ to be the set of positive elements in $(G, \preceq)$ 
and $N$ to be 
the set of negative elements in $(G, \preceq)$, then it is easy to see that 
$\{P,N,\{0\}\}$ is a partition of $G$.

A maximum element $M$ of $(G, \preceq)$ is an 
element such that $x \preceq M$ for all 
$x \in G$. Similarly a minimum element $m$ of $(G,\preceq)$ is an element 
such that $m \preceq x$ for all $x \in G$.
Notice that if there is a maximum element, it is unique as $\preceq$ is a 
linear order and similarly for a minimum element.

An extreme element of $(G,\preceq)$ is either a maximum or a minimum element.
\end{defn}

We collect in the next lemma some basic but useful facts about ordered 
psuedomonoids.

\begin{lem}
\label{lem: basicorder}
Let $(G, \preceq)$ be a pseudomonoid with a weak order. Then: \\
(a) If $G$ possesses a minimum element $m$ then either $m=0$ or 
$m$ is the unique negative element in $(G,\preceq)$. \\
(b) If $G$ possesses a maximum element $M$ then either $M=0$ or 
$M$ is the unique positive element in $(G,\preceq)$. \\
(c) If $G$ possesses a minimum and a maximum element then the order of $G$ 
is less than or equal to 3. \\
(d) If the order of $G$ is infinite, then $G$ possesses at most one extreme 
element. \\
(e) If $G$ is a finite pseudomonoid, then the order of $G$ is either one, two 
or three. Furthermore, for each of these orders, there is a unique 
pseudomonoid up to equivalence.
\end{lem}
\begin{proof}
For (a), let $m$ be a minimum element and assume $m$ is not zero. Then 
we must have $m \prec 0$ as $m$ is a minimum element. 

Suppose there were $x \prec 0$ with $x \neq m$, then 
$x + m \prec 0 +m=m$ with $x + m \in G$ as $G$ is a pseudomonoid. This 
contradicts the minimality of $m$ and thus we conclude there is no such $x$, 
i.e., $m$ is the unique negative element. 

The proof of (b) is similar to (a) and is left to the reader. 
For (c), note that if $G$ has a minimum element $m$ and a maximum element $M$ 
then it follows that the set of nonpositive elements is $\{0,m\}$  by (a) 
and the set of nonnegative elements is $\{0,M\}$ by (b). Thus 
$G=\{0,m,M\}$ and hence $G$ has order less than or equal to 3. 
(Exact order depends 
on whether or not the elements $\{0,m,M\}$ are distinct or not.) 

(d) follows immediately from (c). The first part of 
(e) also follows immediately from (c) since 
any weak order on a finite pseudomonoid has a maximum and a minimum element.

Note that if the order of $G$ is three and $G=\{0,m,M\}$, then we must 
have $M=-m$ since $M+m \in G$.
Thus it is easy to see that $[[G]]=[[\{-1,0,1\}]]$.
If $G$ has order two, obviously $[[G]]=[[\{0,1\}]]$ and if $G$ has order one, 
then $G=\{0\}$. So we are done.
\end{proof}

\begin{rem}
From Proposition~\ref{pro: sc}, we have a complete list of finite 
dimensional, self-centralizing Lie algebras (in the case that $\mathbf{k}$ is 
algebraically closed). The reader can easily verify, 
that each of these is strongly graded, graded by a finite pseudomonoid 
of size 
one, two or three.
\end{rem}

\begin{defn}
Let $\mathfrak{L}=\oplus_{g \in G}E_g$ be a strongly graded Lie algebra and 
suppose we have a weak order $\preceq$ on the pseudomonoid $G$.
Then if $\alpha \in \mathfrak{L}$ is nonzero we can uniquely write 
$$
\alpha = \sum_{i=1}^n e_{g_i}
$$
where $g_1 \prec g_2 \prec \dots \prec g_n \in G$ and 
$e_{g_i} \in E_{g_i}$ is 
nonzero for all $1 \leq i \leq n$.

We call $g_1 \in (G, \preceq)$ the initial index of $\alpha$ and write 
$g_1 = \Init (\alpha)$.

We call $g_n \in (G, \preceq)$ the terminal index of $\alpha$ and write 
$g_n = \Term (\alpha)$.
\end{defn}

\begin{lem}
\label{lem: index}
Let $\mathfrak{L}$ be a strongly graded Lie algebra, 
graded by a weakly ordered 
pseudomonoid $(G,\preceq)$. Then if $x,y$ are nonzero elements of 
$\mathfrak{L}$, we have: \\
\noindent
(a) If $\Term (x) \neq \Term (y)$ then $[x,y] \neq 0$ and 
$$
\Term ([x,y])=\Term (x) + \Term (y).
$$ 
\noindent
(b) If $\Init (x) \neq \Init (y)$ then $[x,y] \neq 0$ and 
$\Init ([x,y])=\Init (x) + \Init (y)$.
\end{lem}
\begin{proof}
$\mathfrak{L}=\oplus_{g \in G}E_g$ so we can take a basis $\{e_g\}_{g \in G}$ 
of $\mathfrak{L}$ with $e_g \in E_g$ for all $g \in G$.
First we expand $x$ in the basis $\{e_g\}_{g \in G}$. Thus 
$$
x = \sum_{i=1}^n x_{g_i} e_{g_i}
$$
with $g_1 \prec g_2 \prec \dots \prec g_n$ and $x_{g_i} \neq 0$ for all 
$1 \leq i \leq n$. Thus $\Init (x)=g_1$ and $\Term (x)=g_n$.

We can expand $y$ in a 
similar manner. 
$$
y = \sum_{j=1}^m y_{h_j} e_{h_j}
$$
with $h_1 \prec \dots \prec h_m$ and $y_{h_j} \neq 0$ all $1 \leq j \leq m$.
Thus $\Init (y)=h_1$ and $\Term (y)=h_m$.

Then we calculate that 
$$
[x,y] = \sum_{i=1}^n \sum_{j=1}^m x_{g_i} y_{h_j} [e_{g_i},e_{h_j}].
$$

Hence, if $g_n \neq h_m$ then $0 \neq [e_{g_n},e_{h_m}] \in E_{g_n + h_m}$ 
and 
$g_n + h_m$ is easily seen to be the 
terminal index of 
$[x,y]$, 
and similarly, if $g_1 \neq h_1$ then $g_1 + h_1$ is the initial index of 
$[x,y]$.
\end{proof}

\begin{cor}
\label{cor: indexeigen}
Let $\mathfrak{L}$ be as in Lemma~\ref{lem: index}. Suppose $\alpha \in 
\mathfrak{L}$ is nonzero. Then: \\
(a) If $\Init (\alpha) \neq 0$ then every eigenvector $x$ of $ad(\alpha)$ has 
$\Init (x) = \Init (\alpha)$. \\
(b) If $\Term (\alpha) \neq 0$ then every eigenvector $x$ of $ad(\alpha)$ has
$\Term (x) = \Term (\alpha)$. \\
(c) $\dime(E_a(\alpha)) \leq 1$ for all $a \in \mathbf{k}$.
\end{cor}
\begin{proof}
For (a), let $\alpha$ have $\Init (\alpha) \neq 0$ and 
assume $x$ is an eigenvector of $ad(\alpha)$ with 
$\Init (x) \neq \Init (\alpha)$. Then by Lemma~\ref{lem: index} we have 
$[\alpha,x]$ is nonzero and 
$$
\Init ([\alpha,x])=\Init (x) + \Init (\alpha).
$$
However, as $x$ is an eigenvector, we also have $[\alpha,x]=\mu x$ for some 
$\mu \in \mathbf{k}$. Since $[\alpha,x] \neq 0$ we conclude $\mu \neq 0$ and 
hence that 
$$
\Init (\alpha) + \Init (x) = \Init ([\alpha, x]) = \Init (\mu x) = \Init (x).
$$
Thus $\Init (\alpha)=0$ which contradicts our hypothesis. Thus we conclude every 
eigenvector of $\alpha$ must have the same initial index as $\alpha$.
The proof of (b) is similar and is left to the reader.

For (c), note that if both $\Init (\alpha)$ and $\Term (\alpha)$ are zero, then 
$\alpha$ is a nonzero scalar multiple of $e_0$ and the result is clear.
So we can assume one of $\Init (\alpha)$ or $\Term (\alpha)$ is nonzero.
For concreteness, let us assume $\Init (\alpha) \neq 0$, the proof for the case 
where $\Term (\alpha) \neq 0$ being similar and left to the reader.

Then if $\dime(E_a(\alpha)) \geq 2$ for some $a \in \mathbf{k}$. We can 
find linearly independent $x,y \in E_a(\alpha)$. By (a), we have 
$\Init (x)=\Init (y)=\Init (\alpha)$. Then it is clear we can form a nonzero linear 
combination of $x$ and $y$ whose $\Init (\alpha)$-component is zero. Call this 
element $z$ then this means that $\Init (z)$ is not $\Init (\alpha)$. 
This is a contradiction as $z$ is nonzero and in $E_a(\alpha)$ 
and so, by (a) again, must have $\Init (z) = \Init (\alpha)$.
\end{proof}

We are now ready to prove an important proposition. 
This proposition will enable 
us to define the spectrum of a strongly graded 
Lie algebra and use it as a tool to distinguish between two 
such Lie algebras.

\begin{pro}
\label{pro: Specwd}
Let $\mathfrak{L}$ be an infinite dimensional, strongly graded Lie algebra, 
graded by a pseudomonoid $G$. Choose a weak order $\preceq$ on $G$ and 
 let $\{e_g\}_{g \in G}$ be the usual basis of $\mathfrak{L}$.

Suppose we have nonzero $\alpha \in \mathfrak{L}$ such that 
$M(\alpha)=\mathfrak{L}$, then:\\
\noindent
(a) If $(G, \preceq)$ 
has no nonzero extreme elements, $\alpha = ke_0$ for some 
nonzero $k \in \mathbf{k}$. Thus $\spec(\alpha)=k\spec(e_0)$ and 
$$
[[\spec(\alpha)]]=[[\spec(e_0)]]=[[G]].
$$ 
\noindent
(b) If $(G, \preceq)$ has a nonzero extreme element $m$, then $m$ 
is unique and 
$$
\alpha = ke_0 + k'e_m
$$ 
for some $k,k' \in \mathbf{k}$ with $k \neq 0$.
Furthermore we still have 
$$
[[\spec(\alpha)]]=[[\spec(e_0)]]=[[G]].
$$
\end{pro}
\begin{proof}
Assume the setup as in the statement of the proposition. 

First note that if $\Init (\alpha) \neq 0$ then Corollary~\ref{cor: indexeigen} 
shows that all the eigenvectors of $\alpha$ have initial index equal to 
$\Init (\alpha)$. However these eigenvectors span $\mathfrak{L}$ as 
$M(\alpha)=\mathfrak{L}$ and so it follows easily that $\Init (\alpha)$ is a 
nonzero minimal element of $(G,\preceq)$.

Similarly if $\Term (\alpha) \neq 0$ then $\Term (\alpha)$ is a nonzero 
maximal element of 
$(G, \preceq)$. 

For (a), note that our previous arguments show that if $(G, \preceq)$ has no 
nonzero extreme elements, that $\Init (\alpha)=0=\Term (\alpha)$ and hence 
that $\alpha = ke_0$ for some nonzero $k \in \mathbf{k}$ from which the rest 
of the conclusion in (a), is obvious.

For (b), note that we can assume that at least one of $\Term (\alpha), 
\Init (\alpha)$ is a nonzero extreme element of $(G, \preceq)$ or else the 
conclusion would follow from our argument for (a).

Since $\mathfrak{L}$ is infinite dimensional, $G$ is infinite and hence 
$(G, \preceq)$ can possess at most one extreme element by 
Lemma~\ref{lem: basicorder}, part (d). 
Thus for (b), we can assume $(G,\preceq)$ has 
exactly one extreme element $m$ and that it is a minimum. 
(If it was a maximum, reorder $G$ by setting $x \prec' y \iff y \prec x$.
This reordering switches $\Init (\alpha)$ and $\Term (\alpha)$ but does not 
change the conclusions of this proposition.)

Thus we have that without loss of generality, 
$\Init (\alpha)=m \prec 0$ 
is the minimum of $(G, \preceq)$ and that $\Term (\alpha)=0$
(Recall if $\Term (\alpha) \neq 0$, we showed before that it would be a 
nonzero maximum which is a contradiction to our assumption). Thus we have 
$$ 
\alpha = k'e_m + T + ke_0
$$
where $k,k' \in \mathbf{k}$ are nonzero and $T$ consists of terms which 
have components corresponding to elements in $g \in G$ which have 
$m \prec g \prec 0$. By Lemma~\ref{lem: basicorder}, part (a), 
there are no such elements $g$, 
and so we conclude that $\alpha = k'e_m + ke_0$. 

It remains to show that
$[[\spec(\alpha)]]=[[G]]$. 
Since $[[\spec(\alpha)]]$ does not change if we scale 
$\alpha$, we will assume from now on that $k=1$. So $\alpha = k'e_m + e_0$.

Suppose $x$ is an eigenvector of $ad(\alpha)$ corresponding to eigenvalue 
$\mu \in \mathbf{k}$ with $0 \prec \Term (x)$. Then $x = ae_{\Term (x)} + D$ 
where $D$ has nonzero components only in indices $g \in G$ with 
$g \prec \Term (x)$, and $a \in \mathbf{k}$ is nonzero. Then 
$$
[\alpha,x] = [k'e_m + e_0, ae_{\Term (x)} + D] 
= a\Term (x)e_{\Term (x)} + D'.
$$
where $D'$ has nonzero components only in indices $g \in G$ with 
$g \prec \Term (x)$.

However $[\alpha, x]=\mu x$ and so we have
$$
a\Term (x)e_{\Term (x)} + D' = \mu a e_{\Term (x)}  + \mu D
$$
from which it follows that $\mu = \Term (x) \in G$.

Now if $x$ is an eigenvector of $ad(\alpha)$ with $\Term (x) \preceq 0$ then 
$x=ce_m + de_0$ and it is easy to check that 
 $x$ must be a scalar multiple of $e_m$ or of $\alpha$ corresponding to the 
eigenvalues $m$ and $0$ respectively. In any case, we 
have $[\alpha,x]=\Term (x) x$.

Thus we see that if $x$ is any eigenvector of $ad(\alpha)$, then $x$ 
corresponds to the eigenvalue $\Term (x) \in G$. So 
$\spec(\alpha) \subseteq \spec(e_0)=G$. Furthermore, $m,0 \in \spec(\alpha)$, 
with $m$ a minimum element of $\spec(\alpha)$ 
under the ordering inherited from $G$.

However, we also see that if $x, y$ are eigenvectors of $ad(\alpha)$ 
corresponding to 
different eigenvalues, then $\Term (x) \neq \Term (y)$ and 
we must have $[x,y] \neq 0$ by 
Lemma~\ref{lem: index}. Since $M(\alpha)=\mathfrak{L}$ and 
$\dime(E_a(\alpha)) \leq 1$ for all $a \in \mathbf{k}$ by 
Corollary~\ref{cor: indexeigen}, we conclude that $\mathfrak{L}$ is strongly 
graded with respect to the eigenspaces of $\alpha$.

Thus reversing the roles of $\alpha$ and $e_0$ in the part of the 
proof where we showed $\spec(\alpha) \subseteq \spec(e_0)$, 
and noting that $e_0 = \alpha - k'e_m$, we conclude that 
$\spec(e_0) \subseteq \spec(\alpha)$ and hence that $\spec(e_0) = \spec(\alpha)$
and thus we are done.
\end{proof}

\begin{defn}
Let $\mathfrak{L}$ be a strongly graded Lie algebra.
We define 
$$
\spec(\mathfrak{L})=[[\spec(\alpha)]]
$$ 
where $\alpha$ is a nonzero 
element in $\mathfrak{L}$ with $M(\alpha)=\mathfrak{L}$.
\end{defn}

Note that $\spec(\mathfrak{L})$ is well-defined if $\mathfrak{L}$ is infinite 
dimensional, by 
Proposition~\ref{pro: Specwd}. 

If $\mathfrak{L}$ is finite dimensional, then 
Corollary~\ref{cor: indexeigen}, part (c), shows that 
$$
\dime(E_a(\alpha)) \leq 1
$$ 
for all $a \in \mathbf{k}$ and so we must have the order of $\spec(\alpha)$ 
is equal to the dimension of $\mathfrak{L}$ for any nonzero $\alpha$ with 
$M(\alpha)=\mathfrak{L}$. Since there is exactly one pseudomonoid of order 
$\spec(\alpha)$ up to equivalence 
by Lemma~\ref{lem: basicorder}, $\spec(\mathfrak{L})$ is 
well-defined in this case also.

We now show that $\spec(\mathfrak{L})$ is truly an invariant of $\mathfrak{L}$.

\begin{pro}
\label{pro: Specinvariance}
Let $\mathfrak{L}, \mathfrak{L'}$ be two Lie algebras and 
$f: \mathfrak{L} \rightarrow \mathfrak{L'}$ be a Lie algebra 
homomorphism. Then: \\
\noindent
(a) For every $\alpha \in \mathfrak{L}$ and $a \in \mathbf{k}$, we have 
$$
f(E_a(\alpha)) \subseteq E_a(f(\alpha)).
$$
Hence $f(M(\alpha)) \subseteq M(f(\alpha))$. \\
\noindent
(b) If $f$ is injective, then $\spec(\alpha) \subseteq \spec(f(\alpha))$. \\
\noindent
(c) If $f$ is bijective, then $\spec(\alpha) = \spec(f(\alpha))$ and 
furthermore 
$$
f(M(\alpha))=M(f(\alpha)).
$$ \\
(d) If $\mathfrak{L}, \mathfrak{L'}$ are two 
strongly graded Lie 
algebras, and $f$ is an isomorphism, 
then $\spec(\mathfrak{L})=\spec(\mathfrak{L}')$.
\end{pro}
\begin{proof}
For (a), notice that if $x \in E_a(\alpha)$, then 
$[\alpha,x]=ax$ and hence 
$$
f([\alpha,x])=af(x).$$
Since $f$ is a Lie algebra homomorphism, 
we have $f([\alpha,x])=[f(\alpha),f(x)]$ and so we conclude 
$[f(\alpha),f(x)]=af(x)$ and thus $f(x) \in E_a(f(\alpha))$.
Also $M(\alpha)=\oplus_{a \in \mathbf{k}}E_a(\alpha)$ and so  
$$
f(M(\alpha))=\oplus_{a \in \mathbf{k}}f(E_a(\alpha)) \subseteq 
\oplus_{a \in \mathbf{k}}E_a(f(\alpha))=M(f(\alpha)).
$$
This gives us (a).

For (b), notice that if $f$ is injective, and we had nonzero 
$x \in E_a(\alpha)$, then $f(x)$ would be nonzero, and by (a), it would lie in 
$E_a(f(\alpha))$. This proves (b).

For (c), notice that since $f$ is bijective, $f^{-1}$ exists and is in fact a 
Lie algebra homomorphism. Thus from (a) and (b) applied to $(f, \alpha)$ and 
$(f^{-1}, f(\alpha))$ we get 
$$
f(M(\alpha)) \subseteq M(f(\alpha)) \text{ and }  
f^{-1}(M(f(\alpha))) \subseteq 
M(f^{-1}(f(\alpha)))
$$ 
giving us $f(M(\alpha))=M(f(\alpha))$.
We also get 
$$
\spec(\alpha) \subseteq \spec(f(\alpha)) \text{ and } \spec(f(\alpha)) 
\subseteq \spec(f^{-1}(f(\alpha)))
$$ 
giving us $\spec(\alpha)=\spec(f(\alpha))$.

For (d), note that $\spec(\mathfrak{L})=[[\spec(\alpha)]]$ for some nonzero 
$\alpha \in \mathfrak{L}$ with $M(\alpha)=\mathfrak{L}$. Since $f$ is an
 isomorphism, we have $f(\alpha)$ is nonzero with 
$$
M(f(\alpha)) = f(M(\alpha))=f(\mathfrak{L})=\mathfrak{L}'.
$$
Hence by Proposition~\ref{pro: Specwd}, we have
$$
\spec(\mathfrak{L}')=[[\spec(f(\alpha))]]=
[[\spec(\alpha)]]=\spec(\mathfrak{L}).
$$
Thus we are done.
\end{proof}

\begin{defn}
Two pseudomonoids $G$ and $G'$ are isomorphic if there is a bijection 
$f: G \rightarrow G'$ such that \\
\noindent
(a) $f(0)=0$ and \\
\noindent
(b) $f(x+y) = f(x) + f(y)$ for all distinct $x,y \in G$.

It is easy to see that if $[[G]]=[[G']]$, then $G$ is isomorphic to $G'$.
\end{defn} 

\begin{ex}
\label{ex: four}
The field $\mathbf{k}$ is a vector space over its characteristic subfield 
$\mathbb{Q}$.
If $\dime_{\mathbb{Q}}(\mathbf{k})=\infty$ then we can find $\mathbb{Q}$-vector 
subspaces $V_n$ of $\mathbf{k}$ of dimension $n$ for every $n \in \mathbb{N}$.
Certainly the $\{V_n\}_{n \in \mathbb{N}}$ 
are a family of nonisomorphic pseudomonoids which are simple pseudomonoids 
by Lemma~\ref{lem: simplePM} as they are abelian groups. 

Thus the construction of Example~\ref{ex: three} gives us a family 
$Witt(A(V_n))$ of simple, strongly graded Lie algebras by 
Corollary~\ref{cor: simplecor}. 

Furthermore since $\spec(Witt(A(V_n)))=[[V_n]]$ we see that 
$$
\{Witt(A(V_n))\}_{n \in \mathbb{N}}
$$ 
is an infinite family of nonisomorphic, simple, generalized Witt 
algebras.
\end{ex}

\begin{ex}
\label{ex: five}
Let $\mathbb{N}$ be the monoid of natural numbers. For every pair of 
relatively prime integers $n,m >1$, we define $M_{n,m}$ to be the submonoid 
of $\mathbb{N}$ generated by $n$ and $m$. It is easy to see that 
$M_{n,m}$ is never simple as a pseudomonoid as 
one can find nontrivial restrictions of 
ideal subsets from $\mathbb{N}$. (See Example~\ref{ex: notsimple}.)
Furthermore $M_{n,m}$ is isomorphic to $M_{n',m'}$ if and only if 
$\{n,m\}=\{n',m'\}$. 

Thus again using the construction of Example~\ref{ex: three}, we get an 
infinite family 
$$
Witt(A(M_{n,m}))_{1<n<m, gcd(n,m)=1}
$$ 
of nonisomorphic, nonsimple, 
generalized Witt algebras. By Proposition~\ref{pro: sc}, all of these Lie 
algebras are semisimple and indecomposable and have no abelian Lie subalgebras 
of dimension greater than one.

In contrast, over an algebraically closed field, 
the only finite dimensional Lie algebra which is indecomposable, 
semisimple and has no abelian Lie subalgebras of dimension greater than one 
is $\mathfrak{sl}_2$. 
\end{ex}

\begin{ex}
\label{ex: six}
$$
\spec(Witt(\mathbf{k}[x]))=[[\{-1,0,1,\dots\}]]
$$ 
and 
$$
\spec(Witt(\mathbf{k}[x,x^{-1}]))=\mathbb{Z}
$$ 
by examples~\ref{ex: one} and 
\ref{ex: two}. These spectra are easily seen not to be isomorphic to 
those discussed in examples~\ref{ex: four} and \ref{ex: five}, and not 
isomorphic to each other of course.
\end{ex}

Thus the following is a list of nonisomorphic generalized Witt algebras:
the classical Witt algebra, the centerless Virasoro algebra, 
$Witt(A(M_{m,n}))$ for relatively prime $m,n >1$ and 
$Witt(A(V_n))$ for $\mathbb{Q}$-vector subspaces $V_n$ of $\mathbf{k}$, 
where $\dime_{\mathbb{Q}}(V_n)=n$ for all $n \in \mathbb{N}$.

Thus, we hope we have conveyed the rich variety of generalized Witt algebras 
available!

In the final section, we verify the Jacobian conjecture 
for a class of generalized Witt algebras. 
That is, we show that under suitable hypothesis, any nonzero 
Lie algebra endomorphism of a generalized Witt algebra is actually an 
automorphism.

\section{The Jacobian conjecture}

A polynomial map $f: \mathbb{C}^n \rightarrow \mathbb{C}^n$ is a map with the 
property  
that each of its components is a complex polynomial in $n$-variables.
Such a map is called invertible if it is bijective, and if its inverse is a
 polynomial map also. 
It is easily seen that an invertible polynomial map has the property that the 
determinant of its Jacobian matrix is a nonzero constant as a function on 
$\mathbb{C}^n$. (See~\cite{Cou}). The classical Jacobian conjecture is that 
the converse is true and remains open for all $n \geq 2$. 

One can ask the following question about 
the classical Weyl algebra in $n$-variables. 
(Defined similarly as we did in the beginning of the paper but using 
$n$-variables instead of one.) Is every nonzero algebra 
endomorphism of a classical Weyl algebra  
actually an automorphism? The answer to this question is unknown for all 
$n \geq 1$. If the statement is true for some $n$, then it implies the 
classical Jacobian conjecture in dimension $n$. (See~\cite{Cou}).

One can generalize to:

\begin{defn}
Given a Lie algebra $\mathfrak{L}$, one says that the Jacobian conjecture 
holds for $\mathfrak{L}$, if every nonzero Lie algebra endomorphism is 
actually an automorphism. 
\end{defn}

Certainly the Jacobian conjecture does not hold for all Lie algebras but 
does hold for finite dimensional, simple Lie algebras.

We will show, among other things that the Jacobian conjecture holds for the 
classical Witt algebra which is the Lie algebra of derivations of the 
classical Weyl algebra where the corresponding conjecture remains open.

One can see immediately, the spectral theory machinery developed earlier 
has a lot to say about this. For example one has:

\begin{cor}
\label{cor: exrigidity}
If $Witt(A)$ is a generalized Lie algebra and $f\partial$ is a nonzero element 
such that $\spec(f\partial) \neq \{0\}$. Then for every injective 
Lie algebra endomorphism $F$  
of $Witt(A)$, one has $F(f\partial)=g\partial$ with $W(g) \leq 1$. 
\end{cor}
\begin{proof}
This follows immediately from Theorem~\ref{thm: spectraltheorem} and 
Proposition~\ref{pro: Specinvariance}.
\end{proof}

Corollary~\ref{cor: exrigidity} shows that the image of an element under 
an injective endomorphism, is reasonably 
constrained by its spectrum. Of course, 
Corollary~\ref{cor: exrigidity} is a rough application of these ideas and 
we will have to refine them a bit to get our desired result.
To this end, we define:

\begin{defn}
\label{defn: self-containing}
A pseudomonoid $G \subseteq \mathbf{k}$ is called self-containing if there is 
nonzero $a \in \mathbf{k}$ such that $aG \subset G$ and $aG \neq G$.
\end{defn}

Notice in this case that $aG$ is a subpseudomonoid of 
$G$ which is equivalent to $G$ so we 
could also define a pseudomonoid to be self-containing 
if it possesses a proper 
subpseudomonoid equivalent to itself.

The integers $\mathbb{Z}=\{\dots,-1,0,1,\dots \}$ 
is an example of a self-containing pseudomonoid 
since $n\mathbb{Z}$ is a proper subpseudomonoid equivalent to 
$\mathbb{Z}$ for all natural numbers $n \geq 2$. The reader can verify that 
this is in fact a complete list of all such proper subpseudomonoids.   

We next give examples of pseudomonoids which are not self-containing.  

\begin{lem}
\label{lem: notselfcontaining}
Any subfield $E$ of $\mathbf{k}$ 
is not a self-containing pseudomonoid. \\
\noindent  
$\{-1,0,1,\dots\} \subseteq \mathbf{k}$ is not a self-containing psuedomonoid.
\end{lem}
\begin{proof}
Suppose $aE \subseteq E$ for some nonzero $a \in \mathbf{k}$. Since 
$1 \in E$, it follows that $a \in E$. Then given $x \in E$, $xa^{-1} \in E$ 
and $x=a(xa^{-1})$. Thus $aE=E$. So $E$ is not self-containing.

Give $G=\{-1,0,1,\dots\}$ the weak order inherited by viewing it as 
the usual subset of the real numbers. If $aG \subseteq G$ for some nonzero 
$a \in \mathbf{k}$, it again follows as $1 \in G$, that $a \in G$.

Clearly $a \neq -1$ so $a > 0$. Then we must have $a(-1)=-a \in G$ and 
hence $-a=-1$ and $a=1$. Thus $aG=G$ and so $G$ is not a self-containing 
pseudomonoid either.
\end{proof}

\begin{defn}
Let $\mathfrak{L}$ be a strongly graded Lie algebra, graded by $G$. Then 
we can write 
$\mathfrak{L}=\oplus_{g \in G} E_g$ 
as usual. For nonzero $x \in \mathfrak{L}$, we let $x_g$ be the $g$-component 
of $x$. 

We define the support of $x$ as  
$$
\Supp (x)=\{g \in G | x_g \neq 0\}.
$$ 

We also define $\Supp (0)=\emptyset$.
\end{defn}

\begin{defn}
A weak order $\preceq$ on a pseudomonoid $G$ is called discrete if 
for every $a,b \in G$, the order of the set 
$\{g \in G | a \preceq g \preceq b \}$ is finite.

A pseudomonoid which possesses a discrete order is called discrete.
\end{defn}

Every subpseudomonoid of the integers is discrete by restricting the standard 
weak order.

We are now ready to prove:

\begin{thm}
\label{thm: Jacobian} 
Let $\mathfrak{L}$ be an infinite dimensional, 
strongly graded Lie algebra, graded by a  
pseudomonoid $G$. Suppose $G$ possesses a discrete order $\preceq$. 

Write $\mathfrak{L}=\oplus_{g \in G}E_g$ as usual and let  
$\{e_g\}_{g \in G}$ be a basis of $\mathfrak{L}$ with the usual properties. 
Let $\Theta$ be the correspondence 
map of Proposition~\ref{pro: idealcor}. 
 
Then for 
every injective Lie algebra endomorphism $f$ of $\mathfrak{L}$,  
we have one of the following two possibilities: \\
\noindent
(a) 
$$
f(e_0)=\frac{1}{a}e_0
$$ 
for some nonzero $a \in \mathbf{k}$ such that 
$aG \subseteq G$. In this case $f(\mathfrak{L})=\Theta(aG)$. Hence if 
$G$ is not self-containing, then $f$ is onto. \\
\noindent
(b) 
$$
f(e_0)= \frac{1}{a}e_0 + D
$$ 
for some nonzero $a \in \mathbf{k}$ 
such that $aG \subseteq G$ and $\Supp (D)$ consists of elements $\prec' 0$.
(Here $\preceq'$ is either equal to $\preceq$, or is $\preceq$ reversed.)
Furthermore there is $\preceq'$-minimal $I \in \Supp (D)$ such that 
$I \preceq' ag$ for all $g \in G$.

In the situation of (b), if $G$ is not self-containing, then 
$I$ is actually a minimum element of 
$(G, \preceq')$, and 
$$
f(e_0) =\frac{1}{a}e_0 + k'e_I.
$$ 
Furthermore $f$ is onto.
\end{thm}
\begin{proof}
Let $f: \mathfrak{L} \rightarrow \mathfrak{L}$ be an injective endomorphism 
of Lie algebras. Then $f(\mathfrak{L})$ is an infinite dimensional
 Lie subalgebra of $\mathfrak{L}$.

Write $\mathfrak{L}=\oplus_{g\in G} E_g$ as in the statement of the theorem  
and let $\preceq$ be a discrete order on $G$. 

Now 
$$
f(\mathfrak{L})=f(M(e_0)) \subseteq M(f(e_0))
$$
and 
$$
\spec(e_0) \subseteq \spec(f(e_0)) 
$$
by Proposition~\ref{pro: Specinvariance}. Thus $f(e_0) \in f(\mathfrak{L})$ 
is $ad$-diagonalizable on $f(\mathfrak{L})$. (In other words, there is a basis 
for $f(\mathfrak{L})$ consisting of eigenvectors of $ad(f(e_0)).)$

Using the chosen order on $G$, we can speak of $I=\Init (f(e_0))$ and 
$T=\Term (f(e_0))$ which both lie in $G$.

By Corollary~\ref{cor: indexeigen}, we conclude that 
if $I \neq 0$ then every eigenvector $x$ of $ad(f(e_0))$ has $\Init (x)=I$.
Similarly, if $T \neq 0$, then every eigenvector $x$ of $ad(f(e_0))$ has 
$\Term (x)=T$.

Let us assume 
both $I$ and $T$ are nonzero to derive a contradiction. 
Let $S=\{g \in G | I \preceq g \preceq T\}$. Since $(G,\preceq)$ is discrete, 
$S$ is finite. Since $I,T$ are nonzero, we have seen that every 
eigenvector of $ad(f(e_0))$ will lie in $\Theta(S)$, and hence 
$f(\mathfrak{L}) 
\subseteq \Theta(S)$ which is a contradiction as $f(\mathfrak{L})$ is 
infinite dimensional.  

So at least one of $I, T$ is zero. By reordering $G$ if necessary, 
we can assume $T=0$. (Notice, if you reverse a discrete order by setting 
$x \prec' y \iff y \prec x$, you get a
 discrete order where $T$ and $I$ interchange. Also notice that this 
reordering  
will not affect the conclusion of the theorem.)

Now if $I=0$ also then $f(e_0)=ke_0$ for nonzero $k \in \mathbf{k}$. 
Now by Proposition~\ref{pro: Specinvariance}, 
$$
G=\spec(e_0) \subseteq \spec(f(e_0))=\spec(ke_0)=kG.
$$
Thus $\frac{1}{k}G \subseteq G$. 
Then notice that $f(E_b(e_0)) \subseteq E_b(ke_0)=E_{\frac{b}{k}}(e_0)$ 
for all $b \in G$ by Proposition~\ref{pro: Specinvariance}. 
Since $E_{\frac{b}{k}}$ is one dimensional, we conclude that 
$f(E_b)=E_{\frac{b}{k}}$ for all $b \in G$ and hence that 
$$
f(\mathfrak{L}) = f(\oplus_{g \in G}E_g)=\oplus_{g \in G}E_{\frac{g}{k}}=
\Theta(\frac{1}{k}G).
$$
So in this case, we get the situation described in (a) of the theorem if 
we set $a=\frac{1}{k}$.

So we may now assume $I \neq 0$, and hence that $I \prec 0$.

Thus $f(e_0) = ke_0 + D$ where every element of $\Supp (D)$ is negative with 
minimum element $I$. 

Now if $x$ is an eigenvector of $ad(f(e_0))$ corresponding to $\mu \in 
\spec(f(e_0))$, we may write:
$$
x=\sum_{i=1}^nx_{g_i}
$$
where $g_1 \prec \dots \prec g_n \in G$ and $x_{g_i} \in 
E_{g_i}$ is nonzero for all $1\leq i \leq n$.

Then a simple calculation shows that 
$$
[f(e_0),x]= kg_nx_{g_n} + D'
$$
where $\Supp (D') \subseteq \{g \in G | g \prec g_n\}$.
Since this must equal $\mu x$, we conclude that 
$kg_n = \mu$ or in other words $k\Term (x)=\mu$.
Thus we conclude that $\spec(f(e_0)) \subseteq k\spec(e_0)$.
However, by Proposition~\ref{pro: Specinvariance}, it follows that 
$\spec(e_0) \subseteq \spec(f(e_0))$.
Thus $G=\spec(e_0) \subseteq \spec(f(e_0)) \subseteq k\spec(e_0)$.
Hence $\frac{1}{k}G \subseteq G$ in this case also. 
 
Now since $I \neq 0$, every eigenvector $x$ corresponding to $\mu$ 
of $f(e_0)$ has $\Init (x)=I$.
Thus $I=\Init (x) \preceq \Term (x)=\mu/k$ and we conclude that 
$I \preceq \frac{g}{k}$ 
for all $g \in G$ since $G \subseteq \spec(f(e_0))$.

Now if $G$ is not self-containing, we must have $\frac{1}{k}G=G$ and 
hence $I$ is a mimimum element of $G$. 
Since $I \prec 0$, it is the unique such 
element. Thus since we had $f(e_0)=ke_0 + D$ where 
$\Supp (D) \subseteq \{g \in G| g \prec 0\}$, we conclude that 
$f(e_0)=ke_0 + k'e_I$. 

Now $kI \in G$ as $\frac{1}{k}G = G$.  
Then by Proposition~\ref{pro: Specinvariance}, we have 
$0 \neq f(e_{kI}) \in E_{kI}(f(e_0))$. 

By our previous analysis, 
$k\Term (f(e_{kI}))=kI$ and so $\Term (f(e_{kI}))=I$. Since $I$ is a minimum of  
$(G,\preceq)$, we conclude $f(e_{kI})$ is a nonzero multiple of $e_I$. 
Hence $e_I \in f(\mathfrak{L})$. 

Since $f(e_0)=ke_0 + k'e_I$ in $f(\mathfrak{L})$, 
we conclude that $f(\mathfrak{L})$ 
contains $e_0$. Now by Proposition~\ref{pro: 
idealcor}, it follows that $f(\mathfrak{L})=\Theta(S)$ where $S$ consists  
of the union of the supports of the elements in $f(\mathfrak{L})$.

However for every $g \in G$, $kg \in G$  and $\Term (f(e_{kg}))=g$ by an  
analysis similar to the one done previously. Hence $S=G$ and f is 
onto. Thus we are done.

\end{proof}

\begin{cor}
\label{cor: classWitt}
Let $\mathfrak{L}$ be a strongly graded Lie algebra, graded by a 
discrete pseudomonoid which is not self-containing. 
Then every injective Lie algebra endomorphism of $\mathfrak{L}$ 
is an automorphism.

If $f$ is any nonzero Lie algebra endomorphism of the classical Witt 
algebra, then $f$ is an automorphism, and furthermore
$$
f(x\partial)=(x+b)\partial
$$
for some $b \in \mathbf{k}$.
Thus the Jacobian conjecture holds for the classical Witt algebra.
\end{cor}
\begin{proof}
The first part follows immediately from Theorem~\ref{thm: Jacobian}.

By Example~\ref{ex: one}, the classical Witt algebra is a strongly graded 
Lie algebra graded by the pseudomonoid $G=\{-1,0,1,\dots\}$ which is obviously 
discrete and is not self-containing by Lemma~\ref{lem: notselfcontaining}.
We have already seen that this Lie algebra is simple, hence any nonzero Lie 
algebra endomorphism $f$  
is injective and hence an automorphism by Theorem~\ref{thm: Jacobian}.

Furthermore, in the strong grading of the classical Witt algebra, we can 
take $x\partial=e_0$ and $x^n\partial \in E_{n-1}$ for all $n \in \mathbb{N}$.

Notice further that if $aG \subseteq G$, in fact $a=1$ as we saw in the 
proof of Lemma~\ref{lem: notselfcontaining}. 
Thus applying Theorem~\ref{thm: Jacobian} again and noting that we must have 
$I=-1$ if we are in situation (b), we conclude furthermore that 
$$
f(x\partial)=(x+b)\partial
$$
for some $b \in \mathbf{k}$.
\end{proof}

\begin{cor}
\label{cor: Virasoro}
If $f$ is a nonzero Lie algebra endomorphism of the centerless Virasoro 
algebra then $f$ is injective and  
$$
f(x\partial)=\frac{1}{a}x\partial
$$
for some nonzero integer $a$.

However, the Jacobian conjecture is false for this Lie algebra. Thus there 
exist injective Lie algebra endomorphisms of the centerless Virasoro algebra 
which are not automorphisms.
\end{cor}
\begin{proof}
By Example~\ref{ex: two}, the centerless Virasoro algebra is strongly 
graded by the pseudomonoid $G=\mathbb{Z}=\{\dots,-1,0,1,\dots\}$, 
with basis $e_n = x^{n+1}\partial \in E_n$ for all $n \in \mathbb{Z}$.
$G$ is obviously discrete.

Let $f$ be a nonzero Lie algebra endomorphism. Since the centerless Virasoro 
algebra is simple, $f$ is injective. 
It is easy to see that $a\mathbb{Z} \subseteq \mathbb{Z}$  if and only if 
$a$ is an integer. Also if we use the standard order of $\mathbb{Z}$, 
then there is no $I$ as in situation (b) of Theorem~\ref{thm: Jacobian}, 
and so we immediately conclude from the same theorem that:
$$
f(x\partial)=\frac{1}{a}x\partial
$$
for some nonzero integer $a$ and $\Img (f)=\Theta(a\mathbb{Z})$.

We will now construct such a Lie algebra endomorphism for every nonzero 
intger $a$. Thus for $a \neq \pm 1$, we obtain injective Lie 
algebra endomorphisms which are not onto.

Define $f_a(e_n)=a^{-(n+1)}e_{an}$ 
for all $n \in \mathbb{Z}$. Certainly this 
defines a vector space endomorphism which is not onto if $a \neq \pm 1$. 

We calculate
\begin{align*}
\begin{split}
[f_a(e_n),f_a(e_m)]&=a^{-(n+m+2)}[e_{an},e_{am}] \\
&=(am-an)a^{-(n+m+2)}e_{a(n+m)} \\
&=(m-n)a^{-(n+m+1)}e_{a(n+m)} \\
&=f_a((m-n)e_{n+m}) \\
&=f_a([e_n,e_m]).
\end{split}
\end{align*}

Hence $f$ is a homomorphism of Lie algebras and we are done.
\end{proof}

This concludes our initial study of generalized Witt algebras. 
One sees that for this family of self-centralizing 
Lie algebras, 
spectral analysis provides a powerful tool to answer basic questions locally.
(On $M(\alpha)$ for nonzero $\alpha \in \mathfrak{L}$.)

We found this extremely useful in the case where $\mathfrak{L}=M(\alpha)$ 
for some nonzero $\alpha$, but it should be possible to push these results 
to the more general case by patching together the local spectra to get some 
sort of global scheme.

\bigskip

\noindent
Dept. of Mathematics \\
University of Wisconsin-Whitewater, \\
Whitewater, WI 53190, U.S.A. \\
E-mail address: namk@uww.edu \\

\bigskip

\noindent
Dept. of Mathematics \\
University of Rochester, \\
Rochester, NY 14627, U.S.A. \\
E-mail address: jonpak@math.rochester.edu \\

\end{document}